\documentclass[journal]{IEEEtran}
\usepackage{amsmath,amsfonts}
\usepackage{amsthm,amssymb}
\usepackage{textcomp}
\usepackage[ruled, linesnumbered]{algorithm2e}
\usepackage{stfloats}
\usepackage{url}
\usepackage{verbatim}
\usepackage{graphicx}
\usepackage{hyperref}

\usepackage{array}
\usepackage{booktabs} 
\newcolumntype{L}{>{$}l<{$}}

\usepackage[backend=bibtex, bibstyle=ieee]{biblatex}
\theoremstyle{plain}
\newtheorem{theorem}{Theorem} 
\newtheorem{proposition}[theorem]{Proposition}
\newtheorem{lemma}[theorem]{Lemma}
\newtheorem{coro}[theorem]{Corollary}
\theoremstyle{definition}
\newtheorem{definition}{Definition}
\newtheorem{obs}[theorem]{Observation}

\newcommand{\cL}{\mathcal{L}}
\newcommand{\te}{{\mathbf{:{\mkern-5mu}t}}} 
\newcommand{\he}{{\mathbf{:{\mkern-5mu}h}}}
\newcommand{\ie}{{\mathbf{:{\mkern-5mu}i}}} 
\newcommand{\rn}{{\mathbf{:{\mkern-5mu}r}}} 
\newcommand{\boldmu}{{\boldsymbol{\mu}}}

\newcommand{\dbar}[1]{{\overrightarrow{#1}}}
\newcommand{\ubar}[1]{{\overline{#1}}}
\newcommand{\dab}{\delta_{a,b}}
\newcommand{\dha}{\delta_{0,a}}
\newcommand{\dhab}{\delta_{0,a,b}}
\newcommand{\da}{\delta_{a}}
\newcommand{\db}{\delta_{b}}
\newcommand{\fu}{\mathfrak{u}}

\DeclareMathOperator{\cont}{Cont}
\DeclareMathOperator{\adme}{AE}

\usepackage[export]{adjustbox}
\usepackage{xcolor}

\newcommand{\com}[1]{} 
\usepackage[normalem]{ulem} 
\definecolor{changecolor}{RGB}{192,64,0}

\begin{document}

\title{A $\mu$-distance for semidirected orchard phylogenetic networks}

\author{
  Gerard Ribas${}^1$
  \and
  Joan Carles Pons${}^1$
  \and
  C\'ecile An\'e${}^{2}$, 
 \thanks{${}^1$ Department of Mathematics and Computer Science,
 University of the Balearic Islands, Spain.}
 \thanks{${}^2$ Department of Statistics, University of Wisconsin - Madison, USA.}
}


\maketitle

\begin{abstract} 
In evolutionary biology, phylogenetic networks are now widely used
to represent the historical relationships between species and population,
when this history includes reticulation events such as hybridization,
gene flow and admixture between populations.
Semidirected phylogenetic networks are appropriate models when
the direction of some edges and the root position are not identifiable
from data.
Comparing semidirected networks is important in many applications.
For rooted and directed networks, a $\mu$-representation was
originally introduced to distinguish tree-child networks,
and has since been extended in two different directions:
to the larger class of orchard directed networks by adding an extra
component that counts paths to reticulations;
and to semidirected networks, through an edge-based variant.
However, the latter does not provide a distance between
semidirected and orchard networks.
We introduce here a new edge-based $\mu$-representation capable of
distinguishing distinct 
orchard binary semidirected networks.
For this class, we provide a reconstruction algorithm
and therefore obtain a true distance
that is computable in polynomial time. 
\end{abstract}


\section{Introduction} \label{sec:intro}

\IEEEPARstart{P}{hylogenetic} networks have become central to the representation
of biological processes such as hybridization, gene flow,
admixture, or horizontal gene transfer,
by which genetic material is transferred non-vertically.
These processes, generally referred to as `reticulations',
can be represented in rooted phylogenetic networks,
which generalize rooted trees.
In these networks, populations or individuals that received genetic materials
from multiple parents are represented as `hybrid' nodes with
multiple incoming edges.
Many inference methods are now available to estimate
phylogenetic networks from large genomic data
\cite[see e.g.][for reviews]{2018degnan,2022HibbinsHahn-review,2025Kong-review}.
Many of these methods infer a semidirected network, which is
the most that may be identifiable under many models,
instead of a rooted and fully directed network
\cite[e.g.][]{2023XuAne_identifiability,2025AABR-galledtreechild,2025Englander-level2-JC,2025Rhodes-circular}.
In semidirected networks, hybrid edges (that go into hybrid nodes)
remain directed. The other edges 
may be undirected.
The exact position of the root is constrained by the direction
of hybrid edges, but is generally unknown, and is not identifiable
under many models.

Multi-rooted networks have recently been considered
\cite{2016Wallbank-fig5multiroots,2019SoraggiWiuf,2022Huber-forestbased,2025MaxfieldXuAne-mu-semidirected}.
These networks may be useful to study groups whose ancestral relationships
of common ancestry may be too difficulty to be recovered reliably,
when primary interest lies in the more recent events of
gene transfer or introgression between these groups.

Comparing phylogenetic networks is fundamental to many endeavors.
For example, measuring the dissimilarity between estimated networks
and the ground truth network is crucial in simulation studies,
to quantify the accuracy of an inference method,
or to compare the accuracy of different methods.
Comparing networks is also fundamental to quantify network variability
within posterior sample of networks from a Bayesian analysis.
Distances between networks in a sample may be used, for example,
to visualize the posterior distribution with distance-based dimension reduction
or detect clustering in this distribution.

Several dissimilarity measures have been proposed for rooted networks,
either of polynomial time complexity
\cite[e.g.][]{2009Cardona-treechild,2009Cardona-nakhlehmetric,2010HusonRuppScornavacca,2010Nakhleh-metric}
or NP-hard
\cite{2023Landry-cherrydistance,2025Marchand-contractions}.
When choosing a dissimilarity for some application,
one has to choose between a polynomial complexity
or the power of discriminating
between any two non-isomorphic networks (with a non-zero dissimilarity).
Indeed, the isomorphism problem
is not generally solvable in polynomial time for rooted network
\cite{2021JanssenMurakami,2014Cardona},
and therefore not for semidirected networks either
(as semidirected networks include rooted networks).

For semidirected networks with a single root (of unknown placement)
\cite{2023LinzWicke} proposed the cut-edge-transfer distance,
but it is NP-hard to compute.
Only one polynomial-time dissimilarity has been proposed
so far for semidirected networks:
\cite{2025MaxfieldXuAne-mu-semidirected}
considered general phylogenetic semidirected networks and
showed their dissimilarity to be a distance in the class of tree-child
(non-binary) semidirected networks.
Their dissimilarity is based on an encoding of the network, called an
edge-based $\mu$-representation, in which each edge contributes an entry
that records the number of paths to each leaf starting from this edge
in either one or both directions.
This encoding was first introduced by \cite{2009Cardona-treechild}
for rooted networks, with one entry per node.
Here we extend the edge-based encoding of \cite{2025MaxfieldXuAne-mu-semidirected}
in two ways. First, we include the number of paths to hybrid nodes
as in \cite{2024Cardona-extendedmu}.
Second, the encoding stores two vectors for any edge that may be
subdivided to place the root,
enriching the approach by \cite{2025MaxfieldXuAne-mu-semidirected}.
We prove that the dissimilarity obtained by comparing our
$\mu$-representation of two networks provides a distance, in the space of
orchard semidirected binary networks.

Orchard networks received a lot of attention in the rooted case
\cite{2019Erdos,2021Bai-ancestralprofile,2021JanssenMurakami,2022vanIerselJanssenJonesMurakami,2026Reichling-murep-semidirected-orchard},
forming a large class that includes so-called tree-child networks.
They are computationally attractive thanks to their
recursive reducibility via tree cherries and reticulate cherries,
which are small local structures adjacent to two leaves.
Reducing these cherries decreases the complexity of
the network, 
which enables inductive arguments, and
efficient reconstruction and comparison algorithms.

In the semidirected context,
\textcite{2026Holtgrefe-multisemidirected} recently defined orchard networks,
with one or multiple roots.
Compared to their work, we mostly focus here on binary networks but
consider a large class of semidirected networks, in which
any edge may have a fixed direction, beyond hybrid edges and those below.
Although semidirected networks are defined differently in
\textcite{2026Holtgrefe-multisemidirected}, our definition of cherries
are equivalent and our orchard definition is similar.
We recall the
differences and similarities between our frameworks where needed.

Section \ref{sec:semidir} defines semidirected multirooted $\cL$-networks,
and their cherries: tree cherries and reticulate cherries.
Orchard $\cL$-networks are defined as those that can be reduced
to isolated notes by a series of cherry reductions.
Section \ref{sec:semi-edge-mu} introduces our edge-based $\mu$-representation
of $\cL$-networks, which is richer than that of
\cite{2025MaxfieldXuAne-mu-semidirected}, in particular by borrowing
the extension by \cite{2024Cardona-extendedmu}.
We define tree and reticulate cherries and their reductions
on edge-based $\mu$-representations directly,
and prove that the operations of reducing cherries and
taking the $\mu$-representation of a network commute.
Finally, Section \ref{sec:classification} presents the main result
that our $\mu$-representation characterizes the $\cL$-network
if it is orchard, with an algorithm to reconstruct the network
from its representation. This $\mu$-representation
leads to a dissimilarity measure between semidirected networks
of polynomial computational cost,
and that is a distance between orchard $\cL$-networks.
To conclude, Section \ref{sec:example} provides examples to illustrate
our algorithm, to show the benefits of our $\mu$-representation
compared to that of \cite{2025MaxfieldXuAne-mu-semidirected},
and counterexamples with non-orchard networks.

\section{Semidirected networks} \label{sec:semidir}

We follow definitions from \textcite{2025MaxfieldXuAne-mu-semidirected}.

\begin{definition}
  A \emph{semidirected graph} $N$ is a tuple $(V,E)$,
  where $V$ is the set of vertices, and $E=E_U \sqcup E_D$ is the set of edges,
  $E_U$ being the set of undirected edges and $E_D$ the set of directed edges.
\end{definition}

If an edge $e\in E$ is incident to nodes $u$ and $v$ in $V$,
we denote $e$ as $uv$ generally.
To indicate directionality or lack thereof, we denote $e$ as $\ubar{uv}$ if $e\in E_U$,
otherwise as $\dbar{uv}$ if $e\in E_D$ is directed from $u$ to $v$. 

A \emph{directed graph} is a semidirected graph with no undirected edges, i.e. $E_U=\emptyset$.

For $v\in V$, $\deg_i(v,N)$ denotes the number of directed edges with $v$ as their child,
$\deg_o(v,N)$, the number of directed edges with $v$ as their parent,
$\deg_u(v,N)$ the number of undirected edges in $N$ incident to $v$, and
$\deg(v)=\deg_i(v)+\deg_o(v)+\deg_u(v)$ the \emph{degree} of $v$.
We may omit $N$ when no confusion is likely.
For an edge $\dbar{uv}$, we say that $u$ is its
parent and $v$ is its child.
Hence $u$ is a parent of $v$ and $v$ is a child of $u$.
Let $v$ be a node in $N$. Then,

\begin{itemize}
  \item $v$ is called a \emph{root} if $\deg_i(v)=\deg_u(v)=0$, i.e. if it only has outgoing edges,
  \item $v$ is called a \emph{leaf} if $\deg_o(v)=\deg_u(v)=0$,
  \item $v$ is called a \emph{tree node} if $\deg_i(v)\le 1$,
  \item $v$ is called a \emph{hybrid node} or a \emph{reticulation node} if $\deg_i(v) > 1$.
\end{itemize}

Furthermore, we shall say that $N$ is \emph{binary} if roots have degree 0, 2 or 3,
leaves have degree 0 or 1,
and the remaining nodes have degree 3.

A \emph{semidirected cycle} in $N$ is a cycle (when ignoring edge directions)
if its undirected edges can be directed so that it becomes a directed cycle.
A semidirected graph is said to be \emph{acyclic} if it does not contain any
semidirected cycle. We refer to directed acyclic graphs as DAGs and to
semidirected acyclic graphs as SDAGs.

A path is a \emph{semidirected path} if its undirected edges can be
directed so that it becomes a directed path. From now on, we only consider
semidirected paths, called `paths' for simplicity.

\begin{figure}[t]
\centering
\includegraphics[scale=1.1,valign=m]{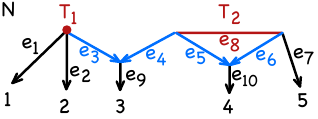}
\\\smallskip 
{\small\begin{tabular}{L|L}
T_i/e_i & \mu\text{-entry} \\
\midrule
T_1    & 111100\rn \\
e_1    & 010000\te, 101100\ie \\
e_2    & 001000\te, 110100\ie \\
e_3    & 100100\he, 011000\ie \\
\midrule
T_2    & 300121\rn \\
e_4    & 100100\he, 200021\ie \\
e_5    & 100010\he, 200111\ie \\
e_6    & 100010\he, 200111\ie \\
e_7    & 000001\te, 300120\ie \\
e_8    & 200110\te, 100011\te \\
\midrule
e_9    & 000100\te \\
e_{10} & 000010\te \\
\end{tabular}}
\caption{
  Top: example of a semidirected $[5]$-network $N$. It has two
  root components $T_1$ and $T_2$, unresolved and highlighted in red.
  $T_1$ is trivial whereas $T_2$ is not. $e_8$ is the unique undirected edge.
  Bottom: $\mu$-representation $\mu(N)$.
  Brackets and parentheses have been removed for readability.
}\label{fig:ex-network}
\end{figure}

Given a set of taxa $\cL$, an \emph{$\cL$-network} $N=(V,E,\varphi)$ is an SDAG
$(V,E)$ whose leaves are tree nodes, and
with an injection $\varphi$ from its leaves to $\cL$.
We do not require $\varphi$ to be a bijection
because we will consider reductions that prune a leaf.
Keeping $\cL$ unchanged, including labels that may be absent
from $N$ will simplify notations.
Unless said otherwise, we will assume that $\cL=[n]=\{1,\ldots,n\}$
for some positive integer $n$,
which corresponds to choosing a fixed ordering of the full set
of leaf labels.

We say $N$ is a \emph{trivial forest} on $\cL_0\subseteq\cL$ if $N$ is the union of
isolated nodes labeled bijectively with $\cL_0$.

As in \cite{2025MaxfieldXuAne-mu-semidirected}, for nodes $u,v\in V$ we
denote $u\sim v$ if there is a path connecting $u$ and $v$ that consists
solely on undirected edges. Note that $\sim$ is an equivalence relation,
and we shall denote $[u]$ the class of $u$ under $\sim$.
The \emph{contraction} of $N$, denoted $\cont(N)$, is the directed graph
obtained by contracting every undirected edge in $N$, which can be seen as the
quotient graph of $N$ under $\sim$, i.e. $\cont(N)=N/{\sim}$,
with vertices $\{[u]|u\in V\}$.
Note that $N/{\sim}$ might not be a binary $\cL$-network (nor a $\cL$-network) even if $N$ is.
We say that $N$ is \emph{complete} if edges incident to leaves are directed (to the leaves)
and $\deg_i([u], N/{\sim})=0$ for every class $[u]$ with more than one node.

From now on, if not said explicitly, we will always consider $\cL$-networks to
be complete and binary, and will be typically denoted as $N$.
The network in Fig.~\ref{fig:ex-network} illustrates the concepts
and notations here and in later sections.

Given a root $[u]$ of $N/{\sim}$, the \emph{root component} associated to $[u]$
is the subgraph of $N$ induced by the nodes of $N$ in $[u]$.
It is a tree with no directed edges \cite{2025MaxfieldXuAne-mu-semidirected}.
Let $T$ be a root component of $N$. If $T$ consists of a single node $u$ in $N$,
which must be a root node then,
we say that $T$ is a \emph{trivial} root component.
If so, we define the \emph{degree} of $T$ as the degree of $u$ in $N$.
If $T$ is a trivial root component of degree 0 or 2,
we say that $T$ is a \emph{resolved} root component.
Otherwise, if $T$ is not trivial or is of degree 3,
$T$ is called an \emph{unresolved} root component.
If $\{u\}=T$ is a root component, we say that $u$ is a resolved
(or unresolved) root if $T$ is resolved (or unresolved, respectively).

Let $u,v$ be nodes and $e$ be an edge in $N$.
We denote by $M_N(u,v)$ the set of paths in $N$ from $u$ to $v$;
and by $M_N(u,v;e) \subseteq M_N(u,v)$ the subset of
paths that do not contain $e$. Also, we denote $m_N(u,v)=|M_N(u,v)|$
and $m_N(u,v;e)=|M_N(u,v;e)|$.
If $N$ is clear from the context, we shall omit the subscript.
See Fig.~\ref{fig:specialretcherry} for an illustration.

\begin{figure}
  \centering\includegraphics[scale=1.1]{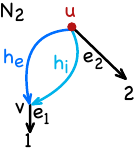}
  \caption{
  Binary $[2]$-network $N_2$, to illustrate counting paths
  $m(u,1; e)$ that avoid an edge $e$.
  For example, $m(u,1; h_e)=1$ from the path through $h_i$,
  and $m(u,1)=2$ in total.
  We will see in section~\ref{sec:semidir} that $(1,2)$ is a
  reticulate cherry and that $N_2$ is an orchard network.
  The $\mu$-entry for edge $h_e$,
  defined later in section~\ref{sec:semi-edge-mu}, is
  $\{((1,1,0), \he),((1,1,1), \ie)\}$.
  }\label{fig:specialretcherry}
\end{figure}

\begin{lemma} \label{lem:bij-paths}
  Let $u,x$  be nodes in an $\cL$-network $N$, which may be non-binary.
  Let $e$ be an edge incident to $u$ which forms a path from $u$ to
  its other node $v$, that is, $e=\ubar{uv}$ or $\dbar{uv}$.
  Then, $M(u,x;e)\sqcup M(v,x;e)$
  is in bijection with $M(u,x)$. In particular, $m(u,x)=m(u,x;e)+m(v,x;e)$.
\end{lemma}

\begin{proof}
  Let $f\colon M(u,x;e)\sqcup M(v,x;e)\longrightarrow M(u,x)$ defined as
  follows. If $P\in M(u,x;e)$ then define $f(P)=P$. Otherwise,
  if $P\in M(v,x;e)$ then define $f(P)=eP$.
  Clearly, $f$ is well-defined and it is a bijection.
\end{proof}

If $T$ is a root component, we define the \emph{set of admissible edges} of $T$,
denoted as $\adme(T)$, as the set of edges in $T$ plus the edges in $N$
adjacent to $T$.

\begin{obs} \label{obs:rc-separated}
  Edges adjacent to (but not part of) a root component must be outgoing.
  Therefore, an edge can be adjacent to at most one root component.
\end{obs}

\subsection{Cherries in $\cL$-networks} \label{subsec:semi-cherries}

\begin{definition}[semidirected cherry]
\label{def:semi-net-cherry}
  Let $N$ be an $\cL$-network and $a\neq b$ non-isolated leaves in $N$.
  Let $p_a$ and $p_b$ be the parents of $a$ and $b$ in $N$.
  We call $(a,b)$ a \emph{tree cherry} in $N$ if $p_a = p_b$.
  We call $(a,b)$ a \emph{reticulate cherry} in $N$ if $p_a$ is a hybrid node
  and if $p_b$ is a parent of $p_a$.
  We say that $(a,b)$ is a \emph{cherry} in $N$, or \emph{reducible} in $N$,
  if it is either a tree cherry or a reticulate cherry.\\
  If $(a,b)$ is a reticulate cherry, let $v$ be the node adjacent to $p_a$
  other than $a$ and $p_b$.
  We say that $\dbar{p_b p_a}$ is the cherry's \emph{internal edge}
  and $\dbar{v p_a}$ is the cherry's \emph{external edge}.
\end{definition}

Note that cherries in semidirected networks have already been defined
by \textcite{2026Holtgrefe-multisemidirected}. In their work, semidirected networks
may also have multiple roots. They are defined slightly less generally than here,
but our definition of cherries is equivalent.
In \cite{2026Holtgrefe-multisemidirected}, semidirected networks must
come from the `deorientation' of a (multi-)rooted network, by which all
tree edges become undirected. Effectively, this means that directed edges
outgoing from root components must be hybrid edges,
whereas we allow them to be tree edges.
In \cite{2026Holtgrefe-multisemidirected}, networks may not have parallel
edges, whereas we do not impose this restriction
(as in $N_2$ from Fig.~\ref{fig:specialretcherry} for example).
Also, in \cite{2026Holtgrefe-multisemidirected} the label set 
of a network may be
strictly larger than its leaf set, as roots may be labelled and have degree 1.
Their definition of `binary' semidirected networks also differs from ours,
as roots must be of degree 1 or 2. However, this difference is only technical,
as they favor `nice' rootings to obtain rooted networks, in which
roots are placed at existing nodes. The exception is that a new root may
subdivide a hybrid edge whose parent is a labelled root, which then becomes
a labelled leaf.
Our framework focuses on edges, with rootings that subdivide edges as in
the traditional tree framework, and with sets of admissible edges (for rooting)
that include directed edges outgoing from a root component.
However, \cite{2026Holtgrefe-multisemidirected} define cherries accordingly,
such that their definition and Definition~\ref{def:semi-net-cherry}
become equivalent.

We will often need to distinguish between different
types of cherries, illustrated in Fig.~\ref{fig:reductions-N}.

\begin{definition}[cherry types]
\label{def:semi-net-cherrytypes}
  Let $(a,b)$ be a cherry in an $\cL$-network $N$,
  and $p_a$, $p_b$
  as in Def.~\ref{def:semi-net-cherry}.
  If $p_b$ is a resolved root, we say that $(a,b)$ is of \emph{type} (r2).
  If $p_b$ is an unresolved root, $(a,b)$ is of type (r3).
  Otherwise, let $e$ be the edge incident to $p_b$ other than $p_b b$
  and $p_b a$ or $p_b p_a$. We say that $(a,b)$ is of type (d)
  if $e$ is directed, and of type (u) if $e$ is undirected.
  We also prepend T or R for tree and reticulate cherries, respectively.
  For example, $(a,b)$ is of type T(r2) if it is a tree cherry of type (r2).
\end{definition}

\begin{definition}[node suppression]
  \label{def:suppress}
  A node $v$ in a $\cL$-network $N$ is said
  \emph{suppressible} if it is not a leaf and either
  \begin{itemize}
    \item $v$ is of degree 2 and $\{v\}$ is not a root component, or
    \item $v$ is of degree 1, $\{v\}$ is a root component, and its child edge
    is a tree edge.
  \end{itemize}
  If $v$ is suppressible, \emph{suppressing} $v$ from $N$ consists in
  deleting $v$ and its incident edges, and in case $v$ was of degree~2,
  connecting its former neighbors $u$ and $w$ by a new edge $uw$.
  If $v$ was incident to at least one directed edge, say $\dbar{vw}$,
  then the new edge is directed as $\dbar{uw}$. Otherwise, the new edge is
  undirected $\ubar{uw}$.
\end{definition}

Note that if $v$ is suppressible of degree 2 and incident
to at least one directed edge, then it is either incident to
$\ubar{uv}$ and $\dbar{vw}$, or to $\dbar{uv}$ and $\dbar{vw}$
(up to relabeling $u$ and $v$) because $N$ is assumed complete.
In this case, defining the new edge $\dbar{uw}$ as directed ensures that
the resulting $\cL$-network is complete.

We are now ready to reduce cherries in $\cL$-networks.

\begin{definition}[semidirected cherry reduction]
  \label{def:semi-cherry-reduce-net}
  Let $(a,b)$ be a cherry in an $\cL$-network $N$.
  The \emph{reduction} of $(a,b)$ in $N$, denoted as $N^{(a,b)}$, is defined as follows.\\
  If $(a,b)$ is a tree cherry in $N$ and $p_a=p_b$ their common parent,
  $N^{(a,b)}$ is the result of deleting
  $a$ and its incident edge and then suppressing  $p_a$ if it is suppressible.\\
  If $(a,b)$ is a reticulate cherry in $N$ and
  if $p_a$ and $p_b$ are their neighbors,
  $N^{(a,b)}$ is the result of deleting $\dbar{p_b p_a}$,
  suppressing $p_b$ if it is suppressible, and then suppressing $p_a$.
\end{definition}

\begin{lemma}
  Let $(a,b)$ be a cherry in an $\cL$-network $N$ (complete and binary).
  Then, $N^{(a,b)}$ is a complete and binary $\cL$-network.
\end{lemma}

\begin{proof}
  Suppressing the nodes that lost an edge ensures that the network remains binary,
  and complete.
  Further, the direction of the new edge(s) when suppressing $p_a$ and/or $p_b$
  ensures that edges incident to leaves in $N^{(a,b)}$ are directed toward leaves.
\end{proof}

If $(a,b)$ is a cherry in $N$ of type (r2), then $p_b$ becomes suppressible and $b$
becomes an isolated node after reducing $(a,b)$.
If $(a,b)$ is of type (r3), then
$p_b$ becomes a resolved root and is not suppressed.

\begin{figure*}
  \centering
  \includegraphics[scale=1,valign=m]{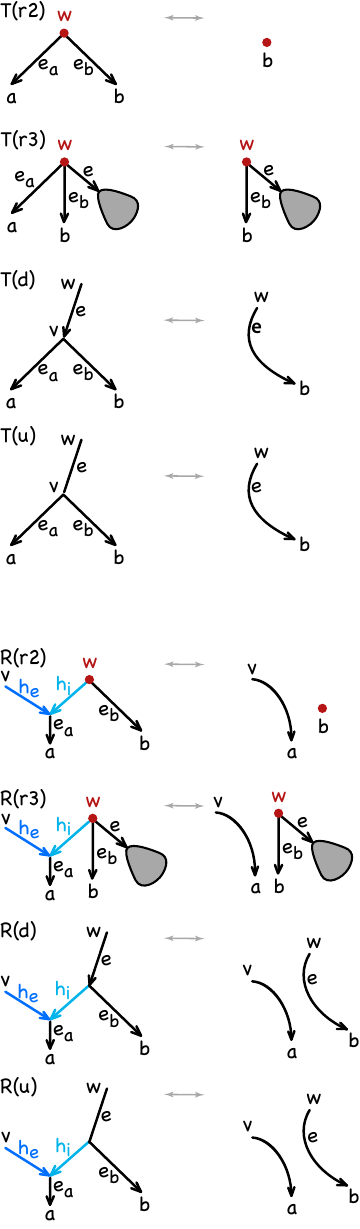}
  \hspace*{.5cm}
  {\small\begin{tabular}{@{}L|LL@{}}
  & \mbox{after cherry} & \mbox{after cherry}\\
  & \mbox{addition} & \mbox{reduction}\\
  \midrule
  \mbox{T(r2)} &&\\
  w   & \dab\rn & \db\rn\\
  e_a & \da\te  & - \\
  e_b & \db\te  & - \\
  \midrule
  \mbox{T(r3)} &&\\
  w   & \dab+m\rn        & \db+m\rn\\
  e_a & \da\te, \db+m\ie & - \\
  e_b & \db\te, \da+m\ie & \db\te \\
  e   &   m\te, \dab\ie  &   m\te \\
  \midrule
  \mbox{T(d)} &&\\
  e   & \dab\te, \text{¿}m\ie? & \db\te, \text{¿}m\ie?\\
  e_a & \da\te & - \\
  e_b & \db\te & - \\
  \midrule
  \mbox{T(u)} &&\\
  e   & \dab\te, m\te & \db\te, m\ie\\
  e_a & \da\te, \db+m\ie & - \\
  e_b & \db\te, \da+m\ie & - \\
  \\\midrule\midrule\\
  \mbox{R(r2)} &&\\
  w   & \dhab\rn       & \db\rn\\
  e_a & \da\te         & - \\
  e_b & \db\te         & - \\
  h_i & \dha\he        & - \\
  h_e & \dha\he, \text{¿}\widetilde{m}\ie? & \da\te, \text{¿}\widetilde{m}\ie?\\
  \midrule
  \mbox{R(r3)} &&\\
  w   & \dhab+m\rn        & \db+m\rn\\
  e_a & \da\te            & - \\
  e_b & \db\te, \dha+m\ie & \db\te  \\
  e   &   m\te, \dhab\ie  &   m\te  \\
  h_i & \dha\he,\db+m\ie  & - \\
  h_e & \dha\he, \text{¿}\widetilde{m}\ie? & \da\te, \text{¿}\widetilde{m}\ie?\\
  \midrule
  \mbox{R(d)} &&\\
  e   & \dhab\te, \text{¿}m\ie? & \db\te, \text{¿}m\ie?\\
  e_a & \da\te            & - \\
  e_b & \db\te            & - \\
  h_i & \dha\he           & - \\
  h_e & \dha\he, \text{¿}\widetilde{m}\ie? & \da\te, \text{¿}\widetilde{m}\ie?\\
  \midrule
  \mbox{R(u)} &&\\
  e   & \dhab\te, m\te    & \db\te, m\ie\\
  e_a & \da\te            & - \\
  e_b & \db\te, \dha+m\ie & - \\
  h_i & \dha\he, \db+m\ie & - \\
  h_e & \dha\he, \text{¿}\widetilde{m}\ie? & \da\te, \text{¿}\widetilde{m}\ie?\\
  \end{tabular}}
  \caption{
  Left: Reduction ($\rightarrow$) and addition ($\leftarrow$) of cherry
  $(a,b)$ of each type.
  In reticulate cherries, the hybrid internal ($h_i$) and external ($h_e$) edges
  are shown in blue.
  Right: $\mu$-entries, as defined later in section~\ref{sec:semi-edge-mu},
  before and after cherry reduction, where $m$ and $\widetilde{m}$ are
  $\mu$-vectors that depend on the rest of the network.
  $\text{¿}m\ie?$ means that the tagged $\mu$-vector may be absent,
  and brackets and parentheses are omitted for readability.
  }\label{fig:reductions-N}
\end{figure*}

Each type of cherry reduction can be reversed by a corresponding
\emph{cherry addition}, similarly to Definition~5 in \textcite{2021JanssenMurakami},
restricted to their addition types 1(a) and 2(a) to build binary networks.

\begin{definition}[semidirected tree cherry addition]
\label{def:semi-addtreecherry}
Let $N$ be an $\cL$-network, and $a,b\in\cL$ such that $b$ is a leaf in $N$
but $a$ is not. For C one of T(r2), T(r3), T(d) or T(u),
the \emph{tree addition} of $(a,b)$ of \emph{type} C in $N$, denoted as
${}^{(a,b)_\mathrm{C}}N$, is the network
obtained by adding to $N$ the taxon $a$ and applying the operations matching
case C listed below.
If the case conditions do not hold, we define ${}^{(a,b)_\mathrm{C}}N = N$.
Let $w$ be the parent of $b$ in $N$, if it exists. 
\begin{enumerate}[T(r2)]
  \item[T(r2)] If $b$ is an isolated leaf in $N$, add a node $p_b$ and directed
        edges $\dbar{p_b a}$ and $\dbar{p_b b}$.
  \item[T(r3)] If $b$ is not isolated and $w$ is a resolved root,
    add one directed edge $\dbar{wa}$.
  \item[T(d)] If $b$ is not isolated and $w$ is not a resolved root,
    add a new node $p_b$, delete $\dbar{wb}$ and add directed
    edges $\dbar{w p_b}$, $\dbar{p_b a}$ and $\dbar{p_b b}$.
  \item[T(u)] If $b$ is not isolated and $w$ is not a resolved root,
    add a new node $p_b$, delete $\dbar{wb}$, and then
    add $\ubar{w p_b}$, $\dbar{p_b a}$ and $\dbar{p_b b}$.
\end{enumerate}
\end{definition}

\begin{definition}[semidirected reticulate cherry addition]
\label{def:semi-addretcherry}
Let $N$ be an $\cL$-network, and $a,b\in\cL$ such that $a\neq b$ are
leaves in $N$. For C one of R(r2), R(r3), R(d) or R(u),
the \emph{reticulate addition} of $(a,b)$ of \emph{type} C in $N$,
denoted as ${}^{(a,b)_\mathrm{C}}N$, is the
network obtained by applying the operations matching case $C$ below
if its conditions are met.
If the conditions are not met, we define ${}^{(a,b)_\mathrm{C}}N = N$.
Let $w$ be the parent of $b$ in $N$, if it exists.
\begin{enumerate}[R(r2)]
  \item[R(r2)] If $b$ is an isolated leaf in $N$, add a node $p_b$ and directed
        edges $\dbar{p_b p_a}$ and $\dbar{p_b b}$.
  \item[R(r3)] If $b$ is not isolated and $w$ is a resolved root,
    add one directed edge $\dbar{w p_a}$.
  \item[R(d)] If $b$ is not isolated and $w$ is not a resolved root,
    add a new node $p_b$, delete $\dbar{wb}$ and add directed
    edges $\dbar{wp_b}$, $\dbar{p_b p_a}$ and $\dbar{p_b b}$.
  \item[R(u)] If $b$ is not isolated and $w$ is not a resolved root,
    add a new node $p_b$, delete $\dbar{wb}$, and then
    add $\ubar{wp_b}$, $\dbar{p_b p_a}$ and $\dbar{p_b b}$.
\end{enumerate}
\end{definition}
 
We shall omit the words ``tree'' and ``reticulate'', and simply say ``addition
of $(a,b)$ of type C''. Note the importance of specifying the type of
addition, because cases T(d) and T(u) (resp. R(d) and R(u)) have
identical conditions. The 8 types of additions are
illustrated in Fig.~\ref{fig:reductions-N}.
It is clear that each type of addition corresponds to
the reverse operation of the reduction of that type.
In other words:

\begin{proposition}\label{prop:reduce-add}
  Let $N$ be an $\cL$-network and $(a,b)$ a cherry of type \textnormal{C} of $N$.
  Then, ${}^{(a,b)_\mathrm{C}}(N^{(a,b)})=N$.
\end{proposition}

A sequence of taxon pairs $S=s_1 \ldots s_k$ is a \emph{reduction sequence} for
an $\cL$-network $N$ if $s_1$ is a cherry in $N$ and if, for
each $i\in\{2,\ldots,k\}$, $s_i$ is a cherry
in $(\ldots(N^{s_1})^{s_2}\ldots)^{s_{i-1}}$. In this
case, $(\ldots(N^{s_1})^{s_2}\ldots)^{s_{k}}$ is called the reduction of $N$
with respect to $S$, and is denoted as $N^S$.
If $S$ is such that $N^S$ is a trivial forest, then we say that $S$ is
\emph{complete} and we call $N$ an \emph{orchard} $\cL$-network.

If an orchard network has multiple cherries, any of them
can be chosen first to reduce the network to a trivial forest.
See the Supplementary Material 
for a give the proof of this result,
which is an adaptation of the results in
\cite{2021JanssenMurakami} for single-rooted networks.

\begin{proposition}\label{prop:order-semi}
  Suppose $N$ is an orchard $\cL$-network and let $(a,b)$ be a cherry of $N$.
  Then, $N^{(a,b)}$ is an orchard $\cL$-network.
\end{proposition}

Next, we show that if an orchard $\cL$-network $N$ has parallel edges,
then they must be adjacent to some unresolved root component.

\begin{proposition}\label{prop:parallel}
  Let $N$ be an orchard $\cL$-network, and suppose that $h_i$ and $h_e$ are
  parallel edges in $N$. Then,
  \begin{enumerate}
    \item If $S=s_1\ldots s_k$ is a complete reduction sequence for $N$, then
      there exists $i\in[k]$ such that $N^{s_1\ldots s_{i-1}}$ contains
       $N_2$, the network in Fig.~\ref{fig:specialretcherry}
      (up to leaf labels) as one of its components,
      and this network contains $h_i$ and $h_e$.
    \item $h_i,h_e\in\adme(T)$ for some unresolved root component $T$ in $N$.
  \end{enumerate}
\end{proposition}

\begin{proof}
  Let $u$ and $v$ be the tail and the head of
  the two parallel edges, labeled $h_i$ and $h_e$
  in an arbitrary but fixed way.
  Note that $u$ is a tree node whereas $v$ is a hybrid node.

  If we reduce $N$ with a complete reduction
  sequence $S=s_1\ldots s_k$, then  $v$ must be suppressed
  eventually during the  process. Let $i\in[k]$ be such that $v$ is a
  node in $\widetilde{N}=N^{s_1\ldots s_{i-1}}$ but it is not
  in $\widetilde{N}^{s_i}$. Then $v$ is a hybrid node in $\widetilde{N}$,
  and it must be involved in the reduction of the cherry $s_i$.
  Therefore, $s_i$ must be a reticulate cherry, involving either
  $h_i$ or $h_e$ (interchangeably), and both $u$ and $v$ must be incident to leaves.
  Thus, $\widetilde{N}$ contains $N_2$ depicted in
  Fig.~\ref{fig:specialretcherry}, as the connected component of $\widetilde{N}$
  containing $h_i$ and $h_e$, proving part 1.

  To prove part 2, we first note that $u$ cannot be a resolved
  root in $N^{s_1\ldots s_{j}}$ for any $j=0,\ldots k$.
  If it were, then $h_i$ and $h_e$ would be its unique incident edges,
  and neither $u$ nor $v$ could be involved in any cherry reduction,
  so the component containing $u,v$ 
  could never be reduced to an isolated node --- a contradiction.

  Let $e$ be the edge incident to $u$ other than $h_i$ and $h_e$,
  and $w$ the node incident to $e$ other than $u$.
  Since $u$ is adjacent to a leaf in $\widetilde{N}$ but not in $N$,
  $w$ must have been suppressed by $s_1\ldots s_{i-1}$.
  When $w$ is suppressed,
  $e$ is deleted altogether if $w$ is a root of degree 1.
  This would cause $u$ to have degree 2, which is not possible,
  as seen earlier. Therefore, when $w$ is suppressed,
  $e$ and the other edge incident to $w$ are replaced by a new edge $e'$.
  If $e$ were directed towards $u$, then $e'$ would also be
  directed towards $u$ (by Def.~\ref{def:suppress}).
  Repeating this argument at each reduction step,
  we get that $u$ would not be adjacent to a leaf in $\widetilde{N}$.
  Thus, $e$ is either undirected or $e$ is directed away from $u$.
  Either way, $u$ is part of a root component $T$,
  and $h_i,h_e\in\adme(T)$, proving part 2.
\end{proof}

\section{Edge-based $\mu$-representation} \label{sec:semi-edge-mu}

We adapt the definitions of \textcite{2024Cardona-extendedmu} to $\cL$-networks.
Let $N$ be an $\cL$-network with leaf set $\cL_0\subseteq\cL=[n]$.
To define an edge-based $\mu$-representation generally,
we do not assume that $N$ is binary here.
For each edge $e$ in $N$ (directed or not) and $i\in[n]$, we define
\[
  \mu_i(e,v)=m(v,i;e)
\]
where $u,v$ are the nodes incident to $e$.
That is, $\mu_i(e,v)$ is the number of paths from $v$ to the leaf $i$ that
do not contain $e$.
By symmetry, $\mu_i(e,u)$ is the number of paths from $u$ to $i$ that
do not contain $e$. Also, as in \cite{2024Cardona-extendedmu}, we extend this
definition to the number of paths from $v$ to hybrid nodes that
do not contain $e$. That is, we define (with some abuse of notation, using $0$
to represent any hybrid node):
\[
  m(v,0;e)=\sum_{h \text{ hybrid}}m(v,h;e),
\]
and, as before,
\[
  \mu_0(e,v)=m(v,0;e).
\]
Then we define the \emph{simple $\mu$-vector}
\[
  \mu(e,v)=(\mu_0(e,v),\ldots,\mu_n(e,v))
\]
with $\mu(e,u)$ defined similarly by symmetry.
A \emph{tagged $\mu$-vector} is a tuple $(\mu,t)$ where $\mu$ is a
simple $\mu$-vector and $t\in\{\te,\he,\rn,\ie\}$ is a \emph{tag}.
The tags $\te,\he$ and $\rn$ were already introduced in
\cite{2025MaxfieldXuAne-mu-semidirected},
and we use the same meaning here, to associate $\mu$-vectors
with tree edges, hybrid edges, or root nodes respectively.
Below we introduce a new tag $\ie$, which will enable us to identify all
the types of cherries.
A \emph{$\mu$-entry} is a multiset of 1 or 2 tagged $\mu$-vectors.

For a subset $A=\{x_1,\ldots,x_k\}\subseteq \{0,1,\ldots,n\}$ we define
the $\mu$-vector $\delta_{x_1,\ldots,x_k}$ as the $(n+1)$-vector
whose $i^\mathrm{th}$ component (indexed from 0) is 1 for all $i\in \{0\}\cup[k]$
and 0 otherwise.

Following \textcite{2025MaxfieldXuAne-mu-semidirected}, we associate each
root component to a $\mu$-entry. For this, we first need a result
for non-isolated root components.

\begin{lemma}
  Let $N$ be an $\cL$-network (not necessarily binary)
  and $T$ a root component of $N$ such
  that $\adme(T)\neq\emptyset$. Then, the vector $\mu(e,v)+\mu(e,u)$
  does not depend on the choice of the edge $e\in\adme(T)$,
  where $u,v$ are the nodes incident to $e$.
\end{lemma}

\begin{proof}
  Let $e_1=u_1v_1\in\adme(T)$ and $e_2=u_2v_2\in\adme(T)$
  be different edges. Without loss of generality,
  we can assume that $u_1$ and $u_2$ are nodes in $T$.
  Then, $e_i$ is either undirected or directed from $u_i$ to $v_i$,
  for $i\in\{1,2\}$.
  By Lemma~\ref{lem:bij-paths} we get
  $
    \mu(e_1,v_1)+\mu(e_1,u_1) = (m(v_1,i;e_1)+m(u_1,i;e_1))_{i=0}^n
    = (m(u_1,i))_{i=0}^n
  $
  and similarly $\mu(e_2,v_2)+\mu(e_2,u_2) = (m(u_2,i))_{i=0}^n$.
  Now, as $u_1\sim u_2$, we can identify the paths in $N$ starting in $u_1$
  with the paths starting with $u_2$, by appending the (unique and undirected) path
  that joins $u_1$ with $u_2$.
  Therefore, $m(u_1,i)=m(u_2,i)$ for all $i\in \{0\}\cup[n]$,
  which completes the proof.
\end{proof}

We can now associate each root component to a $\mu$-entry.

\begin{definition}
  Let $N$ be an $\cL$-network 
  and $T$ a root component of $N$. We define the
  \emph{$\mu$-entry of $T$}, denoted as $\mu(T)$, as follows. If $T$
  consists on an isolated node labeled by $a$,
  then $\mu(T)=\{(\delta_a,\rn)\}$. Otherwise,
  $\mu(T)=\{(\mu(e,v)+\mu(e,u),\rn)\}$ for some 
  $e\in\adme(T)$, where $u,v$ are the nodes incident to $e$.
\end{definition}

Next, we associate each edge of $N$ to a $\mu$-entry.

\begin{definition}
  Let $e$ an edge of $N$, and let $u,v$ be the nodes incident to $e$. The
  \emph{$\mu$-entry of $e$}, denoted as $\mu(e)$, is defined as follows.
  \begin{enumerate}[(d)]
    \item[(d)] If $e=\dbar{uv}$ and $e\notin\adme(T)$ for any unresolved root
          component $T$, then $\mu(e)=\{(\mu(e,v),t)\}$ with tag $t=\he$ if $v$
          is hybrid and $t=\te$ otherwise.
    \item[(i)] If $e=\dbar{uv}$ and $e\in\adme(T)$ for some unresolved root
          component $T$, then $\mu(e)=\{(\mu(e,v),t),(\mu(T)-\mu(e,u),\ie)\}$
          with tag $t=\he$ if $v$ is hybrid and $t=\te$ otherwise.
    \item[(u)] If $e=\ubar{uv}$, then $\mu(e)=\{(\mu(e,v),\te),(\mu(e,u),\te)\}$.
  \end{enumerate}
\end{definition}

Fig.~\ref{fig:ex-network} illustrates these definitions
on a network $N$ (left), listing the $\mu$-entry associated with
each root component and each edge (right).

For a directed admissible edge $e$ in case (i), our definition differs from that
in \cite{2025MaxfieldXuAne-mu-semidirected}, which gives all directed edges a
unidirectional $\mu$-set. Our definition indicates that $e$ is directed and
admissible, thanks to the new tag $\ie$ introduced.
See $\mu(e_1)$ in Fig.~\ref{fig:ex-network} for example.
This tag $\ie$ indicates that the $\mu$-vector should be `ignored'
when reconstructing the semidirected network from its $\mu$-representation,
in the sense that the edge must be directed according to the other $\mu$-vector.
Alternatively, $\ie$ is for edges `incident' to an unresolved root component.
The $\mu$-vector tagged $\ie$ corresponds to the edge going into $T$
after subdividing $uv$ into 2 edges to root $T$ along $uv$.
Our definition has the advantage of identifying the admissible edges,
as exactly those with $|\mu(e)|=2$.

\medskip
Finally, the \emph{edge-based $\mu$-representation} of $N$, denoted as $\mu(N)$,
is the multiset of $\mu$-entries from all edges and root components in $N$.

We say that a (multi)set of $\mu$-entries $\boldmu$ is a \emph{trivial forest}
on $\cL_0\subseteq\cL$ if $\boldmu=\{(\delta_i,\rn) | i \in\cL_0\}$.

Next are basic results linking features of $N$ and $\mu(N)$.

\begin{proposition} \label{prop:trivial-forest-mu}
  A network $N$ is a trivial forest on $\cL_0$ if, and only if,
  $\mu(N)$ is a trivial forest on $\cL_0$.
\end{proposition}

\begin{proof}
  If $N$ is a trivial forest on $\cL_0$, then $N$ has no edges and trivially
  its $\mu$-representation is $\mu(N)=\{(\delta_i,\rn)|i\in\cL_0\}$.
  Conversely, if $\mu(N)=\{(\delta_i,\rn)|i\in\cL_0\}$, implies that $N$ has no
  edges and its taxa is $\cL_0$, and thus it must be the trivial forest on $\cL_0$.
\end{proof}

\begin{lemma} \label{lem:mu-leaves}
  Let $N$ be an $\cL$-network (complete and binary)
  with leaf set $\cL_0\subseteq \cL$.
  For $a\in\cL$, $a$ is a leaf in $N$ if and only if $\da$ has
  multiplicity 1 in $\mu(N)$.
\end{lemma}

\begin{proof}
  Suppose $a$ is a leaf in $N$. If $a$ is isolated then it is a root,
  $\mu(\{a\})=\{(\da,\rn)\}$, and no other $\mu$-entry contains $\da$.
  Otherwise, we have $\da\in\mu(e_a)$, where $e_a=p a$ is the edge incident to $a$.
  Moreover, if there was another $\mu$-entry $\mathfrak{u}_1$ in $\mu(N)$
  such that $\da\in\mathfrak{u}_1$, then there
  would be another edge incident
  to $a$, contradicting that leaves have in-degree 1; or $p$ would be
  incident to a single edge other than $pa$, directed towards $p$,
  contradicting that $N$ is binary.

  Conversely, if $\da$ has multiplicity 1 in $\mu(N)$, then there is at least
  one path to $a$ in $N$, hence $a$ must be a leaf in $N$.
\end{proof}

Moreover, we can use Algorithm~1 in \cite{2025MaxfieldXuAne-mu-semidirected}
to compute $\mu(N)$ in $\mathcal{O}(n|E|)$ time
(recall that $E$ is $N$'s edge set).
This algorithm does not compute
the paths to reticulations nor the tagged $\mu$-vectors with tag $\ie$, but
can be easily adapted to include these modifications without
altering the computational cost.

\subsection{Cherries in edge-based $\mu$-representations}
\label{subsec:semi-cherries-mu}

In this section we adapt the definition of tree and reticulate cherries to the
edge-based $\mu$-representation of (binary) $\cL$-networks.
Here, $\boldmu$ represents a
(multi)set of $\mu$-entries on the supertaxon set $\cL=[n]$.

With some abuse of notation, for
every $\mu$-entry $\{(m_1,t_1),\ldots,(m_k,t_k)\}$ in $\boldmu$
($k=1$ or $2$),
we denote $m_i\in\boldmu$ and $(m_i,t_i)\in\boldmu$.
We say that $m$ (resp. $(m,t)$) has multiplicity $i$ in $\boldmu$ if there
are $i$ entries in $\boldmu$ (counting also multiplicities) containing the
simple $\mu$-vector $m$ (resp. the tagged $\mu$-vector $(m,t)$ for any tag $t$).
Also, given an entry $\mathfrak{u}=\{\mu_1,\mu_2\}$
such that $\mu_1$ is of multiplicity 1 in $\boldmu$,
we say that $\mathfrak{u}$ is \emph{the entry} of $\mu_1$, and we define
\emph{inverse} $\mu_1^{-1}$ of $\mu_1$ as $\mu_2$.
Also, \emph{removing} $\mu_1$ (resp. $\mu_1^{-1}$) is the operation of
changing $\{\mu_1,\mu_2\}$ to $\{\mu_2\}$ (resp. $\{\mu_1\}$).
Note that this is not equivalent to removing the entry $\mathfrak{u}$.
This definition extends to simple $\mu$-vectors with multiplicity 1 in $\boldmu$
in a natural way, ignoring tags.

In $\boldmu=\mu(N)$ from Fig.~\ref{fig:ex-network} for example,
$m=001000$ has multiplicity 1 from the entry
$\mathfrak{u} = \mu(e_2) = \{(001000,\te), (110100,\ie)\}$,
hence $m^{-1}=110100$.
Removing $m$ or $(m,\te)$ from $\boldmu$ means
replacing $\mathfrak{u}$ by $\{(110100,\ie)\}$.

\begin{definition}[tree cherries in $\mu$-representations]
  \label{def:mu-semi-tree}
  $\boldmu$ has a \emph{tree cherry} $(a,b)$ if $\dab$ has multiplicity 1 in $\boldmu$.
  If so, let $t$ be the tag of $\dab$ in $\boldmu$.
  We say that $(a,b)$ if of type T(r2) if $t=\rn$; T(r3) if $t=\ie$;
  T(d) if $t=\te$ and either $\dab^{-1}\not\in\boldmu$ or $(\dab^{-1},\ie)\in\boldmu$;
  and T(u) otherwise.
\end{definition}

\noindent
For example, $(1,2)$ is a T(r3) tree cherry of $\mu(N)$
in Fig.~\ref{fig:ex-network} (right), with $\delta_{1,2}$ found in
the $\ie$-tagged $\mu$-entry from $\mu(e_3)$.

For reticulate cherries, we first define a basic property that
characterizes them when $\boldmu$ is the $\mu$-representation of some network.

\begin{definition} \label{def:mu-semi-ret}
  $\boldmu$ has a \emph{reticulate cherry} $(a,b)$
  if $\dha$ has multiplicity 2 in $\boldmu$, and
  either $\dhab$ has multiplicity 1 in $\boldmu$
  or $\dhab$ has multiplicity 2, each time tagged $\ie$.
  If so, let $t$ be the tag of $\dhab$ in $\boldmu$
  (well-defined as $\ie$ if $\dhab$ has multiplicity 2).
  We say that $(a,b)$ if of type R(r2) if $t=\rn$; R(r3) if $t=\ie$;
  R(d) if $t=\te$ and either $\dhab^{-1}\not\in\boldmu$ or
  $(\dhab^{-1},\ie)\in\boldmu$; and R(u) otherwise.
\end{definition}

\noindent
Note that $\dhab^{-1}$ is well-defined if $t\neq\ie$ as
$\dhab$ is of multiplicity 1 in that case.

For example, $(3,2)$ is an R(r3) reticulate cherry of $\mu(N)$
in Fig.~\ref{fig:ex-network} (right),
with $\delta_{0,3}$ occurring in $\mu(e_3)$ and $\mu(e_4)$,
and $\delta_{0,2,3}$ found in the $\ie$-tagged $\mu$-entry from $\mu(e_1)$.

We say that $(a,b)$ is a \emph{cherry} of $\boldmu$ if it is either a
tree cherry or a reticulate cherry of $\boldmu$. The following results gives the
correspondence of cherries in $N$ and in $\mu(N)$.

\begin{proposition} \label{prop:cherries-mu-N}
  $(a,b)$ is a tree (resp. reticulate) cherry
  of type C in $\mu(N)$
  if and only if $(a,b)$ is a tree (resp. reticulate) cherry of type C in $N$.
\end{proposition}

\begin{proof}
  This proof is illustrated in Figure~\ref{fig:reductions-N}.
  Let $\boldmu = \mu(N)$.
  First, suppose that $\dab\in\mathfrak{u}$ for some
  $\mu$-entry $\mathfrak{u}$ of $\boldmu$.
  We consider the following cases, which correspond to the cherry type
  of $(a,b)$ in $\boldmu$ if we assumed that $\dab$ had multiplicity 1 in $\boldmu$.
  \begin{enumerate}[T(r2)]
    \item[T(r2)] $\mathfrak{u} = \{(\dab,\rn)\}$.
      Then, there exists a root component $T$ such that, from $T$,
      there is only one path to $a$, another to $b$, and no paths to
      any other leaf or reticulation. Therefore, $T$ must be only
      adjacent to the nodes $a$ and $b$, and thus it must be a resolved root
      component $\{w\}$ such that $w$ is the parent of both $a$ and $b$,
      so $(a,b)$ is a tree cherry in $N$ of type T(r2).

    \item[T(r3)] $(\dab,\ie) \in \mathfrak{u}$.
      Let $e$ be an edge such that $\mathfrak{u} = \mu(e)$
      and let $w$ be the tail of $e$, which must be of degree 3
      (from $N$ being binary, and $\dab$ being tagged $\ie$) 
      Then, from $w$ and without passing
      through $e$, there is only one path to $a$, one to $b$, and no path
      to any other leaf or reticulation. Therefore, $w$ is
      adjacent to both $a$ and $b$, because $w$ is of degree 3
      and $N$ is binary. 
      Therefore $w$ is a root (unresolved) and
      $(a,b)$ is a tree cherry of type T(r3).

    \item[T(d)]  $\mathfrak{u} = \{(\dab,\te)\}$ or $\mathfrak{u} = \{(\dab,\te), (\dab^{-1},\ie)\}$
      for some $\dab^{-1}$. 
      Similarly, consider 
      an edge $e$ such that $\mathfrak{u} = \mu(e)$. Then, $e$ is directed,
      and let $w$ be the head of $e$.
      Again, looking from $w$ and without passing through $e$,
      we conclude that $(a,b)$ is a tree cherry of type T(d).

    \item[T(u)] Otherwise, we have that $\mathfrak{u} = \{(\dab,\te), (\dab^{-1},\te)\}$.
    Indeed, $\dab$ cannot be tagged by $\he$ because any hybrid edge
    leads to at least 1 reticulation. Then, this case is
    analogous to T(d), but now $e$ is undirected and thus $(a,b)$ is a
    tree cherry of type T(u).
  \end{enumerate}
  In particular, if $(a,b)$ is a tree cherry in $\boldmu$,
  then it is a cherry in $N$ of the same type.

  Conversely, assume that $(a,b)$ is a tree cherry in $N$. Let C be its type
  and $p$ be the parent of both $a$ and $b$. If $p$ is a root of degree 2, then
  the only edges incident to $p$ are $\dbar{pa}$ and $\dbar{pb}$,
  hence $\mu(\{p\})=\{(\dab,\rn)\}$. Notice that there cannot be
  another $\mu$-entry containing $\dab$, because the cherry is isolated from the
  rest of the network.
  If C is not T(r2), let $e$ be the edge incident to $p$ other than $pa$
  and $pb$. Then, from $p$ and avoiding $e$, we have only one path to $a$, one
  to $b$, and no path to any other leaf or reticulation.
  Therefore, $(\dab,t)\in\mu(e)$ for some tag $t$. If C is T(r3), then $p$ is an
  unresolved root, and thus $t=\ie$. If C is T(d), then $e$ is directed
  towards $p$, so $t=\te$ and, if $\dab^{-1}$ exists, its tag must be $\ie$.

  If $\dab$ had multiplicity greater than 1 in $\boldmu$, then there must exist
  another edge $e'\neq e$ such that $\dab\in\mu(e')$. The single paths
  from $e'$ to $a$ and $b$ must pass through $p$, so they must also pass
  through the tree edge $e$. 
  Therefore, $e'$ and $e$ must be connected with a path
  consisting on nodes of degree two, contradicting that $N$ is binary.

  \medskip

  For reticulate cherries, we proceed analogously. Suppose
  that $\dhab\in\fu$ for some $\mu$-entry $\fu\in\boldmu$, and $\dha\in\boldmu$.
  From $\dha\in\boldmu$ we get that the parent of $a$ is a reticulation.
  As for tree cherries, but considering 
  the cases corresponding to the cherry type if
  we assumed that $\dha$ and $\dhab$ satisfied the properties in Def.~\ref{def:mu-semi-ret}
  and replacing $\dab$ by $\dhab$, we conclude that if $(a,b)$ is a reticulate
  cherry in $\boldmu$, then it is a reticulate cherry in $N$ of the same type.

  Conversely, assume that $(a,b)$ is a reticulate cherry in $N$, of some type C.
  Below, we show that $(a,b)$ is a cherry of type C in $\boldmu$.
  Let $p_a$ and $p_b$ be the parents of $a$ and $b$, respectively, and
  $h_i$ and $h_e$ the internal and external edges of $(a,b)$. As
  they are hybrid edges incident to $p_a$, then $\dha\in\mu(h_i)$
  and $\dha\in\mu(h_e)$, so $\dha$ has multiplicity at least 2 in $\boldmu$.
  Then $\dha$ has multiplicity exactly 2, because otherwise there
  would be another edge incident to $p_a$ other than $h_i,h_e$ and
  $p_aa$, contradicting that $N$ is binary.
  If C is R(r2), $p_b$ is a root of degree $2$. Then
  the only edges incident to $p_b$ are $\dbar{p_bp_a}$ and $\dbar{p_bb}$,
  so $\mu(\{p\})=\{(\dhab,\rn)\}$. Since $a$ can only be reached from $p_a$
  or $p_b$ and $b$ can only be reached from $p_b$, and $N$ is binary,
  $\dhab$ has multiplicity 1 and $\dha$ has multiplicity 2 in $\boldmu$.
  If C is not R(r2), let $e$ be
  the edge incident to $p_b$ other than $p_bp_a$ and $p_bb$.
  Then $(\dhab,t)\in\mu(e)$ for some tag $t$.
  What remains to be shown is that $\dhab$ has the appropriate
  multiplicity and tag(s) in $\boldmu$.

  If C is R(d) then $e$ is directed towards $p_b$
  and if C is R(u) then $e$ is undirected. In both cases $t=\te$, and
  $e$ is the unique edge such that $\dhab\in\mu(e)$, using that $N$ is binary.
  Then $\dhab$ has multiplicity 1, and
  if C is R(d) and $\dhab^{-1}$ exists, its tag must be $\ie$.
  Finally,
  if C is R(r3), then $p_b$ is an unresolved root,
  $e$ is directed away from $p_b$, $t=\ie$, and we have 2 cases:
  either $e$ has child $p_a$ like $h_i$, or $e$ has a different child.
  In the first case $e=h_e$ and $h_i$ are (exchangeable) parallel edges,
  both with the same $\mu$-entry $\fu=\{(\dha,\he),(\dhab,\ie)\}$.
  Then the subgraph induced by $p_b$, $p_a$, $a$ and $b$
  forms a connected component of $N$ (isomorphic to $N_2$
  in Fig.~\ref{fig:specialretcherry} up to leaf labels)
  and no other edge has $\dhab$ in its $\mu$-entry.
  Thus $\dhab$ has multiplicity 2 in $\boldmu$, each time tagged $\ie$.
  In the second case, $h_e$ is not parallel to $h_i$.
  Then similarly to the R(r2) case, $e$ is the unique edge
  such that $\dhab\in\mu(e)$, so $\dhab$ has multiplicity 1 in $\boldmu$.

  Therefore, $(a,b)$ is a reticulate cherry of type C in $\boldmu$.
\end{proof}

For example, $N$ and $\mu(N)$ in Fig.~\ref{fig:ex-network} have five cherries:
$(1,2)$ of type T(r3), $(2,1)$ of type T(r3),
$(3,1)$ of type R(r3), $(3,2)$ of type R(r3), and $(4,5)$ of type R(u).

\smallskip

As a corollary of the proof above,
we can characterize the reticulate cherries in $\mu(N)$
that correspond to reticulate cherries in $N$ whose
internal and external hybrid edges are exchangeable parallel edges,
and this lets us define $\dhab^{-1}$
when $\dhab$ is of multiplicity 2.

\begin{coro}\label{coro:mu-parallel}
  Let $(a,b)$ be a reticulate cherry in $N$, and let $h_i$ and $h_e$ be its
  internal and the external hybrid edges.
  The edges $h_i$ and $h_e$ are parallel if, and only if,
  $\dhab$ has multiplicity 2 in $\mu(N)$.
  In this case, $(a,b)$ is of type R(r3) and the
  two $\mu$-entries containing $\dhab$ are
  both equal to $\{(\dha,\he),(\dhab,\ie)\}$,
  and so we define $\dhab^{-1} = \dha$.
\end{coro}

For a reticulate cherry $(a,b)$ in $\boldmu$,
if $\boldmu=\mu(N)$ comes from a network $N$ then we will seek to match
each $\mu$-entry containing $\dha$ to one of the two hybrid edges
into the parent of $a$, that is, either to the internal or the external edge
of the reticulate cherry in $N$.

\begin{definition}
  Let $(a,b)$ be a reticulate cherry in $\boldmu$.
  We define the \emph{internal $\mu$-entry} of $(a,b)$ with respect to $\boldmu$ as
  \begin{itemize}
    \item $\{(\dha,\he)\}$ if $\dhab$ has multiplicity 1 in $\boldmu$ and
          either $\dhab^{-1}\not\in\boldmu$ or $(\dhab^{-1},\ie)\in\boldmu$,
    \item $\{(\dha,\he),(\dhab+\dhab^{-1}-\dha,\ie)\}$ otherwise.
  \end{itemize}
\end{definition}

\begin{proposition}\label{prop:distinguishable-ret}
  Suppose $(a,b)$ is a reticulate cherry of $N$.
  Then, the $\mu$-entry of the internal edge of $(a,b)$
  is equal to the internal $\mu$-entry of $(a,b)$ with respect to $\mu(N)$.
\end{proposition}

\begin{proof}
  Let $p_a$ and $p_b$ be the parents of $a$ and $b$ in $N$, respectively,
  and $h_i=p_bp_a$ the internal edge of $(a,b)$.
  Let $e$ be the edge incident to $p_b$ other than $p_bb$ and $h_i$.
  Let $\boldmu = \mu(N)$.
  By Proposition~\ref{prop:cherries-mu-N} we have that $(a,b)$ is a reticulate
  cherry in $\boldmu$. We now prove that the internal $\mu$-entry of $(a,b)$
  with respect to $\boldmu$ is
  precisely $\mu(h_i)$, so in particular it is in $\boldmu$.
  We consider the following cases (see Fig.~\ref{fig:reductions-N}, bottom).
  \begin{enumerate}[R(r2)]
  \item[R(r2)] If $p_b$ is a resolved root, then
    $\mu(h_i)=\{(\dha,\he)\}$.
    Note that in this case, the $\mu$-entry for $\dhab$ is $\{(\dhab,\rn)\}$.
  \item[R(r3)] If $\{p_b\}$ is an unresolved root component $T$, then
    $\mu(h_i)=\{(\dha,\he),(\mu(T)-\dha,\ie)\}$. Now, the edge
    $e$ is in $\adme(T)$, has $(\dhab,\ie)$ in its $\mu$-entry,
    and therefore $\mu(T)=\dhab+\dhab^{-1}$.
    Note that this case 
    includes the possibility that $e=h_e$, when $h_i$ and $h_e$ are parallel.
  \item[R(u)] If $e$ is undirected, it is in an unresolved root component $T$,
    $h_i \in \adme(T)$
    and, as before, $\mu(h_i)=\{(\dha,\he),(\dhab+\dhab^{-1}-\dha,\ie)\}$.
    Also, $\dhab^{-1}$ has tag $\te$ from $e$.
  \item[R(d)] If $e$ is directed,
    then $h_i$ is directed and not admissible,
    so $\mu(h_i)=\{(\dha,\he)\}$.
    In this case, $\mu(e)$ is either $\{(\dhab,\te)\}$ or
    $\{(\dhab,\te), (\dhab^{-1},\ie)\}$.
  \end{enumerate}
  In all cases, $\mu(h_i)$ is precisely $(a,b)$'s internal $\mu$-entry in $\mu(N)$.
\end{proof}

Recall that if $(a,b)$ is a reticulate cherry in $N$, then $\dha$ has
multiplicity 2 in $\boldmu=\mu(N)$.
Indeed, the two $\mu$-entries containing $\dha$ are the $\mu$-entries of the
internal and external edges of $(a,b)$,
as seen in the proof of Prop.~\ref{prop:cherries-mu-N}.
By Prop.~\ref{prop:distinguishable-ret}, the $\mu$-entry of the internal edge
is the internal $\mu$-entry of $(a,b)$.
Therefore, we define the \emph{external $\mu$-entry} of $(a,b)$
with respect to $\boldmu$
as the other $\mu$-entry in $\boldmu$ that contains $\dha$.
By Props.~\ref{prop:cherries-mu-N} and \ref{prop:distinguishable-ret},
it is the $\mu$-entry of the external edge of $(a,b)$. Thus, we
can identify the internal and the external edge of a reticulate cherry
in $\mu(N)$. Note that the internal and external $\mu$-entries might be
equal (as in $N_2$ in Fig.~\ref{fig:specialretcherry}).

For example, the internal $\mu$-entry of $(3,2)$ in the network $N$ of
Fig.~\ref{fig:ex-network}
is
$\{(\delta_{3,2},\he),(\delta_{0,3,2}+\delta_{0,3,2}^{-1}-\delta_{0,3},\ie)\}=\{(\delta_{3,2},\he),(\delta_{1,2},\ie)\}$,
which is from $e_3$, the internal edge of $(3,2)$. Then, the
external $\mu$-entry of $(3,2)$ is $\mu(e_4)$. For the reticulate
cherry $(4,5)$, the internal and external $\mu$-entries are equal
to $\{(\delta_{0,4},\he),((2,0,0,1,1,1),\ie)\}$.

\smallskip

Before defining the reductions of cherries in $\boldmu$, we need to
identify some other $\mu$-entries involved in the cherries.
Let $(a,b)$ be a (tree or reticulate) cherry of $N$. Recall that, by
Lemma~\ref{lem:mu-leaves}, $\da$ and $\db$ exist with multiplicity 1 in $\boldmu=\mu(N)$.
Also, by the proof of Prop.~\ref{prop:cherries-mu-N}, we have that
$\db^{-1}$ exists in $\boldmu$ if $(a,b)$ is of type T(r3) or R(r3),
and that $\dab^{-1}$ (resp. $\dhab^{-1}$) exists if $(a,b)$ is of type T(u) (resp. R(u)).
We are now ready to define reductions of cherries in $\boldmu$.

\begin{definition}
  Let $\boldmu=\mu(N)$ for some $N$.
  Let $a,b\in\cL$ with taxa $\cL$ listed in the order $(a,b,\ldots)$.
  For a tree cherry $(a,b)$ of $\boldmu$, the \emph{reduction} of $(a,b)$
  in $\boldmu$, denoted as $\boldmu^{(a,b)}$, is the multiset obtained by applying
  to $\boldmu$ the following operations in this order.
  \begin{enumerate}
  \item {If $(a,b)$ is of type T(r2) or T(d), remove the entry of $\delta_b$.\\
    If $(a,b)$ is of type T(r3), remove $\delta_b^{-1}$ and $\dab$.\\
    If $(a,b)$ is of type T(u), remove the entry of $\delta_b$ and
    change the tag of $\dab^{-1}$ to $\ie$.}
  \item Remove the entry of $\delta_a$.
  \item Change each simple $\mu$-vector $(m_0,m_a,m_b,m_3,\ldots)$
    to $(m_0,0,m_b,m_3,\ldots)$.
  \end{enumerate}
\end{definition}

Using the distinction between the internal and external $\mu$-entries,
we are able to
define the reduction of reticulate cherries in $\mu(N)$.

\begin{definition}
  Let $\boldmu=\mu(N)$ for some $N$.
  Let $a,b\in\cL$ with taxa $\cL$ listed in the order $(a,b,\ldots)$.
  For a reticulate cherry $(a,b)$ of $\boldmu$,
  the \emph{reduction} of $(a,b)$ in $\boldmu$, denoted as $\boldmu^{(a,b)}$, is
  the multiset obtained by applying to $\boldmu$ the following operations in this order.
  \begin{enumerate}
    \item If $(a,b)$ is of type R(r2) or R(d), remove the entry of $\delta_b$.\\
    If $(a,b)$ is of type R(r3), remove $\delta_b^{-1}$ and
      one tagged $\mu$-vector\footnotemark $(\dhab,\ie)$.\\
    If $(a,b)$ is of type R(u), remove the entry of $\delta_b$ and
    change the tag of $\dhab^{-1}$ to $\ie$.
    \item Remove the $\mu$-entry of $\delta_a$.
    \item Remove the internal $\mu$-entry of $(a,b)$
    (reduce its multiplicity from 2 to 1
     if the internal and external entries are the same).
    \item Change the tag of the unique $(\dha,\he)$ remaining to $\te$,
    \item Change each simple $\mu$-vector
    $(m_0,m_a,m_b,m_3,\ldots)$ to $(m_0-m_a,m_a-m_b,m_b,m_3,\ldots)$.
  \end{enumerate}
\end{definition}
\footnotetext{
  If $\dhab$ has multiplicity 2 in $\boldmu$,
  by Corollary~\ref{coro:mu-parallel} both
  $\mu$-entries containing $\dhab$ are equal,
  so it does not matter from which one $(\dhab,\ie)$ is removed.
}

\begin{proposition} \label{prop:cherry-mu}
  Let $(a,b)$ be a cherry in $N$. Then,
  \( \mu(N^{(a,b)})=\mu(N)^{(a,b)} \).
\end{proposition}

\begin{proof}
  Let C be one of the types T(r2), T(r3), T(d) and T(u).
  Suppose $(a,b)$ is a tree cherry in $N$ of type C. Then
  $(a,b)$ is a tree cherry in $\mu(N)$ of the same type C, and
  the reduction of $(a,b)$ in $N$
  corresponds to the reduction of $(a,b)$ in $\mu(N)$.
  The last operation of the reduction of $(a,b)$ in $\mu(N)$
  corresponds to reducing the number of paths to reticulations and to leaves
  for the edges that are outside the cherry,
  similar to \textcite[Prop. 9]{2024Cardona-extendedmu}.

  An analogous argument can be used for reticulate cherries of each type.

  This correspondence is illustrated in Fig.~\ref{fig:reductions-N}, left.
\end{proof}

Let $S=(a_1,b_1)\ldots(a_k,b_k)$ be
such that $(a_1,b_1)$ is a cherry of $\boldmu$ and $(a_i,b_i)$ is a cherry of
$(\ldots \boldmu^{(a_1,b_1)}\ldots)^{(a_{i-1},b_{i-1})}$ for $i\in\{2,\ldots,k\}$.
We call $S$ a \emph{reduction sequence for $\boldmu$}, and we denote
$\boldmu^S=(\ldots \boldmu^{(a_1,b_1)})\ldots)^{(a_k,b_k)}$.
If $S$ is such that $\boldmu^S$ is a trivial forest, we say that $S$ is \emph{complete}
and that $\boldmu$ is an \emph{orchard} $\mu$-representation.
The following result is obtained by iteratively applying Prop.~\ref{prop:cherry-mu}.

\begin{coro}
  Let $S$ be a reduction sequence for $N$. Then,
  \[ \mu(N^{S})=\mu(N)^{S}.\]
\end{coro}

In section~\ref{sec:example} we provide an example to demonstrate the previous
result for a complete reduction sequence $S$.

\section{Encoding and comparison of orchard $\cL$-networks}
\label{sec:classification}


In this section, we present Algorithm~\ref{alg:reconstruction} to construct an
orchard $\cL$-network given an orchard $\mu$-representation $\boldmu$.
Then we prove that, if $N$ is orchard and if $\boldmu=\mu(N)$,
Algorithm~\ref{alg:reconstruction} reconstructs $N$.
In doing so, we prove that if two orchard networks have the same
edge-based $\mu$-representation then they are isomorphic. In other words,
the $\mu$-representation is an invariant for the class of orchard $\cL$-networks.
This encoding then leads to a dissimilarity between $\cL$-networks
(binary or not)
that is a distance in the class of orchard (binary) $\cL$-networks.

\begin{algorithm}
  \caption{reconstruct($\boldmu$)} \label{alg:reconstruction}

  \KwData{An orchard $\mu$-representation $\boldmu$.}
  \KwResult{A $\cL$-network $N_0$ such that $\mu(N_0)=\boldmu$.}

  \tcp{Completely reduce $\boldmu$:}

  $\boldmu_0\gets\boldmu$\;
  $k\gets0$\;
  \While{$\boldmu_k$ is not a trivial forest}{
     Let $(a_{k+1},b_{k+1})$ be a cherry of $\boldmu_k$, and $T_{k+1}$ its type\;
     $\boldmu_{k+1}\gets\boldmu_k^{(a_{k+1},b_{k+1})}$\;
     $k\gets k+1$\;
  }

  \tcp{Construct $N_0$ from all $\boldmu_i$, $(a_i,b_i)$ and $T_i$:}

  $N_k\gets$ trivial forest associated with $\boldmu_k$\;
  \While{$k>0$}{
     Let $N_{k-1}$ be obtained by adding to $N_k$ cherry $(a_k,b_k)$ of type $T_k$\;
     $k\gets k-1$\;
  }

  \Return{$N_0$}\;
\end{algorithm}

\begin{theorem} \label{thm:alg-unique}
  Let $\boldmu$ be an orchard edge-based $\mu$-representation.
  Then the reduction loop (lines 3-7) of Algorithm \ref{alg:reconstruction} terminates,
  and the $\cL$-network $N$ obtained from Algorithm \ref{alg:reconstruction} is
  orchard and such that $\mu(N)=\boldmu$.
  Further, if $\mu(N_1)=\mu(N_2)$ is orchard, then $N_1$ and $N_2$ are isomorphic.
\end{theorem}

\begin{proof}
  First, note that thanks to Proposition~\ref{prop:order-semi},
  the choice of cherry $(a_{k+1},b_{k+1})$ on line~4 does not matter,
  in the sense that if $\boldmu_k$ has multiple cherries, then $\boldmu_{k+1}$
  will remain orchard regardless of the choice.
  Therefore, the reduction loop will terminate
  whichever cherry is chosen, at iterations when  $\boldmu_k$ has multiple cherries.

  We prove the theorem by strong induction on the number of non-root $\mu$-entries
  in $\boldmu$. If $\boldmu$ has 0 non-root entries, then it is a trivial forest,
  the reduction loop has 0 iterations, and the theorem statement holds.

  Now assume that the result holds for all orchard $\mu$-representations with up to
  $k$ non-root $\mu$-entries.
  Let $\boldmu$ be orchard with $k+1$ non-root entries, and let $(a_1,b_1)$ be the cherry
  chosen in the first iteration of the loop on line 4. Then $\boldmu_1$ is orchard
  by Proposition~\ref{prop:order-semi} and has $k$ or fewer non-root entries.
  Therefore, by induction, the network $N_1$ created on line~10 is orchard and such that
  $\mu(N_1)=\boldmu_1$.
  Let $N_0$ be the network returned by the algorithm.
  Since $(a_1,b_1)$ is a cherry of $\boldmu_0$, by Proposition~\ref{prop:cherry-mu} we get
  that $\mu(N_0^{(a_1,b_1)})=\boldmu_0^{(a_1,b_1)}$, which is $\boldmu_1$ by construction.
  As $N_0$ and $\boldmu$ are obtained by adding to $N_1$ and $\boldmu_1$ 
  the cherry $(a_1,b_1)$ of the same type $T_1$, we get $\mu(N_0) = \boldmu$.
  Further, for any network $N$ such that $\mu(N) = \mu(N_0)$, we have that
  $(a_1,b_1)$ is a cherry of the same type in both $N$ and $N_0$ 
  by Proposition~\ref{prop:cherries-mu-N}.
  Hence $N$ and $N_0$ are isomorphic, as both are obtained
  by adding $(a_1,b_1)$ of type $T_1$ to $N_1$, which is unique up to isomorphism by induction.
\end{proof}

\begin{obs}
  Algorithm \ref{alg:reconstruction} can be used to determine if an
  arbitrary $\cL$-network $N$ is orchard, by applying the
  reduction loop (lines 3-7) to $\boldmu=\mu(N)$.
  Indeed, $N$ is not an orchard $\cL$-network if, and only if, at some iteration
  $\boldmu_k$ is not a trivial forest but does not contain any cherry
  (from Proposition~\ref{prop:order-semi}).
\end{obs}

As a direct consequence of Theorem~\ref{thm:alg-unique}, we get the following
fundamental result.

\begin{coro} \label{coro:iso}
  Let $N_1$ and $N_2$ be orchard $\cL$-networks. Then, $N_1\cong N_2$ if, and
  only if, $\mu(N_1)=\mu(N_2)$.
\end{coro}

\bigskip

Finally, we introduce the edge-based dissimilarity to compare $\cL$-networks using
the edge-based $\mu$-representation.

\begin{definition}
  Let $N_1$ and $N_2$ be (not necessarily binary) $\cL$-networks.
  Their \emph{edge-based $\mu$-dissimilarity} is
  \[
    d_\mu(N_1,N_2)=|\mu(N_1)\triangle \mu(N_2)|
  \]
  where $\triangle$ denotes the symmetric difference of multisets.
\end{definition}

\begin{theorem}
  For a fixed supertaxon set $\cL$, $d_\mu$
  is a distance on the class of (binary) orchard $\cL$-networks.
\end{theorem}

\begin{proof}
  By its definition, $d_\mu$ is non-negative and symmetric.
  Since it is defined from the cardinality of the symmetric difference of multisets,
  it also satisfies the triangle inequality.
  Therefore, it only remains to prove the separation property.

  Let $N_1$ and $N_2$ be orchard $\cL$-networks.
  If $d_\mu(N_1,N_2)=0$, then $\mu(N_1)=\mu(N_2)$.
  By Corollary~\ref{coro:iso}, this implies that $N_1\cong N_2$.
  Conversely, if $N_1\cong N_2$ then clearly $\mu(N_1)=\mu(N_2)$, hence
  $d_\mu(N_1,N_2)=0$.
  Therefore, $d_\mu$ is a distance on the class of orchard $\cL$-networks.
\end{proof}

The distance $d_\mu$ can be computed in polynomial time.
Indeed, by the discussion after Lemma~\ref{lem:mu-leaves}, the edge-based
$\mu$-representation of a network $N$ can be computed in $\mathcal{O}(n|E|)$ time,
where $n=|\cL|$ and $E$ is the edge set of $N$.
Then, after equipping the set of $\mu$-entries with any total order, one can sort
the multisets $\mu(N_1)$ and $\mu(N_2)$ and compute their symmetric difference
by a linear scan.
Therefore, if $k_e=\max\{|E_1|,|E_2|\}$,
where $E_1$ and $E_2$ are the edge sets of $N_1$ and $N_2$ respectively,
the comparison of two $\cL$-networks
takes $\mathcal{O}\bigl(k_e(n+\log k_e)\bigr)$ time.

\section{Examples}%
\label{sec:example}

\begin{figure*}[!h]
  \centering
  \includegraphics[scale=1.1,valign=m]{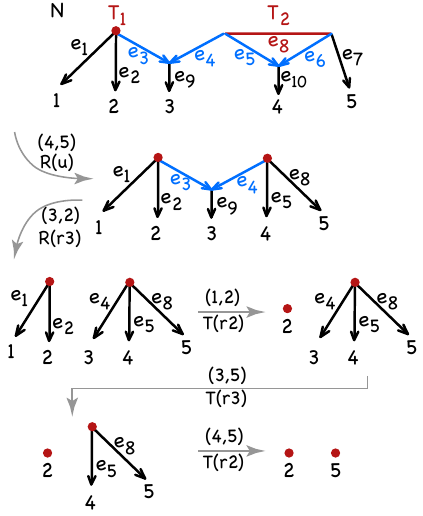}
  \hspace{.5cm}
  {\scriptsize
  \begin{tabular}{@{}L||L@{}}
            & \boldmu_0=\mu(N)       \\
    \midrule
    T_1    & 111100\rn               \\
    e_1    & 010000\te, 101100\ie    \\
    e_2    & 001000\te, 110100\ie    \\
    e_3    & 100100\he, 011000\ie    \\
    \midrule
    T_2    & 300121\rn               \\
    e_4    & 100100\he, 200021\ie    \\
    e_5    & 100010\he, 200111\ie    \\
    e_6    & 100010\he, 200111\ie    \\
    e_7    & 000001\te, 300120\ie    \\
    e_8    & 200110\te, 100011\te    \\
    \midrule
    e_9    & 000100\te               \\
    e_{10} & 000010\te               \\
  \end{tabular}
  \bigskip

  \begin{tabular}{@{}L||L|L|L|L|L@{}}
    & \boldmu_1 & \boldmu_2 & \boldmu_3 & \boldmu_4 & \boldmu_5 \\
    \midrule
    T_1    & 111100\rn                             & 011000\rn & 001000\rn                   & 001000\rn      & 001000\rn \\
    e_1    & 010000\te, 101100\ie                  & 010000\te & -                           & -              & - \\
    e_2    & 001000\te, 110100\ie                  & 001000\te & -                           & -              & - \\
    e_3    & 100100\he, 011000\ie                  & - & -                           & -              & - \\
    \midrule
    T_2    & 100111\rn                             & 000111\rn & 000111\rn                   & 000011\rn      & 000001\rn \\
    e_4    & 100100\he, 000011\ie                  & 000100{\color{blue!80}\te}, 000011\ie & 000100\te, 000011\ie        & 000010\te      & - \\
    e_5    & 000010{\color{blue!80}\te}, 100101\ie & 000010\te, 000101\ie & 000010\te, 000101\ie        & 000001\ie      & - \\
    e_6    & -                                     & - & -                           & -              & - \\
    e_7    & -                                     & - & -                           & -              & - \\
    e_8    & 100110\te, 000001\te                  & 000110\te, 000001\te & 000110\te, 000001\te        & -              & - \\
    \midrule
    e_9    & 000100\te                             & - & -                           & -              & - \\
    e_{10} & -                                     & - & -                           & -              & - \\
  \end{tabular}
  }
  \caption{\label{fig:ex-network-reduction}
    Top: Reduction of $N$ 
    by the complete reduction sequence $S=(4,5)(3,2)(1,2)(3,5)(4,5)$.
    Bottom: $\boldmu_1=\boldmu_0^{(4,5)},\ldots, \boldmu_5=\boldmu_4^{(4,5)}$
    constructed by the cherry reduction loop in Algorithm~\ref{alg:reconstruction}
    from $\boldmu_0 = \mu(N)$.
    The tags that are modified with respect to the previous $\mu$-representation
    are highlighted in blue.
    The addition loop then constructs $N_k$ such that
    $\boldmu_k=\mu(N_k)$.
    Each cherry is indicated with its type near grey arrows.
    Note that the type of each cherry is determined from $\boldmu_k$
    during the reduction, and is used by Algorithm~\ref{alg:reconstruction}
    to perform cherry additions.
  }
\end{figure*}

We first illustrate the $\mu$-representation
and Algorithm~\ref{alg:reconstruction} on the
$\cL$-network $N$ from Fig.~\ref{fig:ex-network} (left) on 5 taxa,
reproduced in Fig.~\ref{fig:ex-network-reduction} (top left).
Its $\mu$-representation $\mu(N)$ is shown in 
Fig.~\ref{fig:ex-network-reduction} (top right).
$N$ has 2 trees cherries, $(1,2)$ and $(2,1)$ of type T(r3);
and 3 reticulate cherries, $(3,1)$ and $(3,2)$ of type R(r3), and $(4,5)$ of type R(u).
Fig.~\ref{fig:ex-network-reduction} shows the reduction of $N$ by a
complete reduction sequence $S$, starting with $(4,5)$.

\noindent
Fig.~\ref{fig:ex-network-reduction} (bottom) 
shows the reduction of $\boldmu_0=\mu(N)$
by Algorithm~\ref{alg:reconstruction} (loop on lines 3-7)
using the same sequence $S$ as in Fig.~\ref{fig:ex-network-reduction} (top left),
leading to a trivial forest with leaves~2 and~5.
To reconstruct $N$, the algorithm starts with the trivial forest
and adds the cherries from $S$, each one
according to the type recorded
during the reduction loop.

\bigskip

Our next example illustrates the need to associate edges that are
directed yet admissible for rooting the semidirected network
with a bidirectional $\mu$-entry, in which one tagged $\mu$-vector is as
in \textcite{2023XuAne_identifiability} and the other $\mu$-vector is tagged $\ie$.
The $\cL$-network $N$ in Fig.~\ref{fig:counterex-noitag} (left) is orchard.
The $\cL$-network $N'$ obtained from $N$ by swapping taxon labels $2$ and $4$
is not isomorphic to $N$ (for example, $1$ and $2$ form a reticulate cherry
in $N$ but $1$ and $4$ do not). Accordingly, $\mu(N)\neq\mu(N')$, but are very
similar: every edge that cannot reach either $2$ or $4$, and whose incident
nodes cannot reach $2$ or $4$, have the same $\mu$-entry in both $N$ and $N'$.
$N$ and $N'$ also have the same set of 3 root $\mu$-entries.
Only 4 $\mu$-entries in $N$ distinguish $N$ from $N'$: $d_\mu(N,N')=8$.
They are from the 4 hybrid edges adjacent to the parents of $2$ and $4$.
As these are unresolved roots, their child edges have bidirectional $\mu$-entries
with $\ie$-tagged $\mu$-vectors (see Fig.~\ref{fig:counterex-noitag} right).
If their $\ie$-tagged components were removed, then the multiset of
$\mu$-entries would become identical between $N$ and $N'$,
showing that the $\ie$ tag is necessary for
Theorem~\ref{thm:alg-unique} and Corollary~\ref{coro:iso} to hold.

Note that this network $N$ provides an example illustrating that our definition
of orchard networks corresponds to `weakly' orchard (binary) networks as
defined by \textcite{2026Holtgrefe-multisemidirected}.
Indeed, by their Theorem~7, weakly orchard networks are exactly
those that admit a complete sequence of cherry reductions
(in their context of semidirected networks without parallel edges,
possibly non-binary).
Here, subdividing the edges marked with brown circles gives
a rooted orchard network, 
while subdividing the edges marked with orange triangles gives a rooted network
that is not orchard (without any cherry). Therefore $N$ is weakly but
not strongly orchard as per \textcite{2026Holtgrefe-multisemidirected}.

\begin{figure}[t]
\centering
\includegraphics[scale=1.1,valign=m]{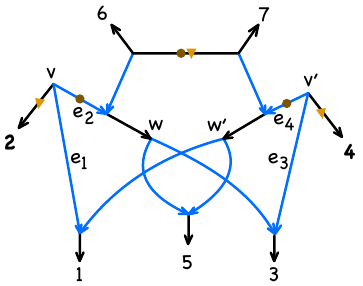}

\smallskip
{\small\begin{tabular}{@{}L|L@{}}
\mathrm{edge} & \mu\text{-entry} \\
\midrule
e_1 & 11{\bf 0}0{\bf 0}000\he, 30{\bf 1}1{\bf 0}100\ie\\
e_3 & 10{\bf 0}1{\bf 0}000\he, 31{\bf 0}0{\bf 1}100\ie \\
\midrule
e_2 & 30{\bf 0}1{\bf 0}100\he, 11{\bf 1}0{\bf 0}000\ie \\
e_4 & 31{\bf 0}0{\bf 0}100\he, 11{\bf 0}0{\bf 1}000\ie \\
\end{tabular}}
\caption{
 Top: 
 $[7]$-network $N$ reduced to the trivial forest
 on $\{1,2,4\}$ by the cherry-reduction sequence
 $(1,2)(3,4)(5,1)(1,4)(3,5)(5,2)(5,6)(6,7)(7,1)$. 
 The network $N'$ obtained from $N$ by swapping taxon labels $2$ and $4$
 is not isomorphic to $N$.
 With our $\mu$-representation that includes $\ie$ tags, $\mu(N)\neq\mu(N')$
 although are similar: they have the same multiset of tagged $\mu$-vectors
 when the components with the $\ie$ tag are removed.
 Subdividing edges with brown circles leads to a rooted, orchard network; while
 subdividing edges with orange triangles gives a rooted non-orchard network.
 \\
 Bottom: 
 $\mu$-entries of the 4 edges in $N$ that do not match any $\mu$-entry
 in $N'$. Brackets and parentheses have been removed for readability.
 Coordinates corresponding to $2$ and $4$ are in bold.
 Here $N$ has 3 root components and 1 undirected edge. Removing the 5 edges
 drawn closest to the top (and suppressing degree-2 nodes)
 provides a smaller, rooted network (with 2 unresolved roots),
 and simpler non-isomorphic orchard networks 
 with the same $\mu$-representation after removing $\ie$-tagged entries.
}\label{fig:counterex-noitag}
\end{figure}

\bigskip

Finally, Fig.~\ref{fig:non-orchard} shows a $[6]$-network $N$, which is not orchard as it does not
have any cherry.
Swapping leaf labels $5$ and $6$ in $N$ gives a network $N'$, also not orchard.
$N$ and $N'$ are not isomorphic but share the same edge-based $\mu$-representation.
This example is adapted from \cite[Fig.~3]{2024Cardona-extendedmu}
on rooted networks,
first introduced in \cite[Fig.~5]{2009Cardona-treechild}
for node-based $\mu$-representations.
To see that $N$ and $N'$ are not isomorphic,
consider for example that there is a path of 3 edges to leaf $1$
starting from the parent of $6$, but not from the parent of $5$.
Also, note that this example provides new cases of rooted non-orchard
networks $N^{+}\not\cong N'^{+}$ with the same node-based $\mu$-representation.
We can obtain $N^{+}$ from $N$ by subdividing any admissible edge of
$N$'s root component $T$ to place the root.
Then, we use that the node-based $\mu$-representation is contained
in $\mu(N^{+})$, by choosing the simple $\mu$-vectors
associated with the edge directions  in $N^{+}$ and ignoring
the tags. Doing so, we obtain 9 example pairs of rooted networks
(of distinct shapes, up to relabelling the leaves), non-orchard, that are
non-isomorphic yet non-distinguishable from their node-based $\mu$-representation.

\begin{figure}[t]
  \centering
  \includegraphics[scale=1,valign=m]{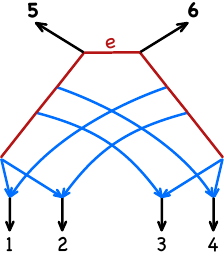}
  \caption{$[6]$-network $N$ without any cherry, hence not orchard.
  The network $N'$ obtained by swapping taxa $5$ and $6$ 
  is not isomorphic to $N$. Yet $N$ and $N'$ have the
  same edge-based $\mu$-representation.
  In fact, each $\mu$-entry is invariant to swapping $5$ and $6$.
  For example, the root component (red edges) has simple
  $\mu$-vector $(8,2,2,2,2,1,1)$,
  and $\mu(e) = \{((4,1,1,1,1,1,0),\te),((4,1,1,1,1,0,1),\te)\}$.
  }
  \label{fig:non-orchard}
\end{figure}

\section*{Code availability}

The \texttt{Julia} package \texttt{PhyloNetworks}
\cite{2017SolislemusBastideAne_PhyloNetworks}
implements simple and tagged $\mu$-vectors
in version 1.1.0.
Function \texttt{mudistance\_semidirected}
implements an edge-based $\mu$-distance
from an edge-based $\mu$-representation
in which $\mu$-vectors include the number of paths $m_0$ to hybrid nodes
\cite[as in][]{2024Cardona-extendedmu},
but unidirectional $\mu$-entries with a single 
$\mu$-vector for edges adjacent to a root component
\cite[as in][]{2025MaxfieldXuAne-mu-semidirected}.
Future development will add an option to use bidirectional $\mu$-entries
for these edges, with an extra $\ie$-tagged $\mu$-vector, to match the
$\mu$-representation described here.

\section*{Acknowledgements}
\noindent
We thank Mark Jones for contributions to initial discussions,
in particular to the example network in Fig.~\ref{fig:counterex-noitag}.
This work was supported in part by the National Science Foundation through grants
DMS-2023239 to C.A., and
DMS-1929284 while the authors visited 
the Institute for Computational and Experimental Research in Mathematics in
Providence, RI, during the Quantitative Phylogenomics semester program.
JCP was supported in part from grant PID2021-126114NB-C44 funded
by MCIN/AEI/10.13039/501100011033.

\printbibliography

\renewcommand{\appendixname}{Supplementary Material}
\appendix[Proof of proposition \ref{prop:order-semi}]
\renewcommand{\thefigure}{A\arabic{figure}}
\label{sec:appendix-JM}

To prove Proposition~\ref{prop:order-semi}, we adapt
the proof of \textcite[Prop. 1]{2021JanssenMurakami},
the analogous result for rooted networks. Their arguments only rely on the
structure of cherries (not on the rest of the network).
This structure is the same here in semidirected networks,
because edges incident to the cherry leaves and the hybrid internal
edge in semidirected reticulate cherries are in fact directed.
Thus, their proof adapts smoothly to our case,
for semidirected (and binary) networks.
We reproduce it here for completeness, along with
the necessary lemmas, as a supplement for lack of novelty.


\begin{lemma}[{\cite[Lemma 6]{2021JanssenMurakami}}]
  \label{lem:order-6}
  Let $N$ be an $\cL$-network on $\cL_0\subseteq\cL$ and let be $(a,b)$ a cherry
  in $N$. Then, the set of cherries in $N$ that are not
  in $N^{(a,b)}$ is contained in $(\{a\}\times\cL_0)\cup(\cL_0\times\{a\})$.
\end{lemma}

\begin{proof}
  Let C be the type of $(a,b)$ in $N$, and let $(c,d)$ be a cherry in $N$ that
  is not a cherry in $N^{(a,b)}$.
  For the sake of contradiction,
  suppose that $c\neq a$ and $d\neq a$. Then adding the cherry $(a,b)$ of type C
  to $N^{(a,b)}$ may only create a parent edge 
  or subdivide the parent edge to $b$,
  and either adds leaf $a$ or subdivides the parent edge to $a$.
  Neither operation changes the fact that $(c,d)$ is not a cherry,
  a contradiction.
\end{proof}

\begin{lemma} [{{\cite[Lemma 7]{2021JanssenMurakami}}}]
  \label{lem:order-7}
  Let $S=s_1\ldots s_k$ be a complete reduction sequence for an
  orchard $\cL$-network $N$ such that $s_i=(a,b)$ is a tree cherry
  in $N^{s_1\ldots s_{i-1}}$.
  Then, $S'=s_1\ldots s_{i-1}(b,a)s_{i+1}'\ldots s_k'$ is a complete reduction
  sequence for $N$, where $s_j'$ is obtained by replacing $b$ by $a$ in $s_j$
  for each $j\in\{i+1,\ldots,k\}$.
\end{lemma}

\begin{proof}
  Let $\widetilde{N}=N^{s_1\ldots s_{i-1}}$. As $(a,b)$ is a tree cherry
  in $\widetilde{N}$, 
  $\widetilde{N}^{(a,b)}$ is equal to 
  $\widetilde{N}^{(b,a)}$ if $b$ is replaced by $a$. Then
  as $\tilde S=s_{i+1}\ldots s_k$ is a complete reduction sequence
  for $\widetilde{N}^{(a,b)}$ and $a$ does not appear in $\tilde S$.
  Therefore, replacing $b$ by $a$ in $\tilde S$ gives a complete reduction sequence
  for $\widetilde{N}^{(b,a)}$.
\end{proof}

\begin{lemma} [{{\cite[Lemma 8]{2021JanssenMurakami}}}]
  \label{lem:order-8}
  Let $S=s_1\ldots s_k$ be a complete reduction sequence for an
  orchard $\cL$-network $N$ such that $s_2$ is a cherry in $N$. Then, $N^{s_2}$
  is orchard.
\end{lemma}

\begin{proof}
  Note that $s_1=(a,b)$ and $s_2=(c,d)$ are cherries in $N$ by assumption.
  Let $p_a,p_b,p_c$ and $p_d$ be the parents of $a,b,c$ and $d$
  respectively (which must exist).
  We consider the following cases:
  \begin{itemize}
    \item $s_1=s_2$ (e.g. in $N_2$ from Fig.~\ref{fig:specialretcherry} with parallel edges).
    Then, $N^{s_1}=N^{s_2}$ is orchard.
    \item $s_2=(b,a)$. Then, both $(a,b)$ and $(b,a)$ are cherries in $N$, and thus
          they must be tree cherries in $N$. Therefore, the leaf $a$ is not present
          in $N^{s_1}$, hence $s_2$ is not a cherry in $N^{s_1}$ and $S$ is not
          a reduction sequence: this case cannot occur.
    \item $c=a$ and $d\neq a,b$. Then, $(a,b)$ and $(a,d)$ are reticulate cherries
          in $N$,
          $p_a$ is a hybrid node and its two hybrid incident edges
          are $p_bp_a$ and $p_dp_a$. Consequently, reducing $s_1$ and then $s_2$
          is the same as reducing $s_2$ and then $s_1$. Therefore,
          as $N^{s_1s_2}=N^{s_2s_1}$ is orchard, we have that $N^{s_2}$ is
          orchard too.
    \item $d=a$ and $c\neq a,b$.
          As $s_2=(c,a)$ is a cherry in $N$, $p_a$ must be a tree node, and then
          $s_1=(a,b)$ must be a tree cherry. But then $a$ is absent
          from $N^{s_1}$ and $s_2$ cannot be a cherry in $N^{s_1}$: a contradiction.
          Therefore this case cannot occur.
    \item $c=b$ and $d\neq a,b$. As $s_1=(a,b)$ is a cherry in $N$, $p_b$ is a
          tree node. Then, as $s_2=(b,d)$ is a also a cherry in $N$, it must be
          a tree cherry and $p_b=p_d$. Therefore, $p_b$ is a root of degree 3.
          By reducing $(a,b)$ and then $(b,d)$ we obtain an isolated leaf labelled $d$.
          If instead we first reduce $s_2=(b,d)$, then $p_b$ is left
          unsuppressed in $N^{s_2}$, so $(a,d)$ is a cherry in $N^{s_2}$
          of the same type as $(a,b)$ in $N$. Reducing
          it we obtain the same network as before:
          $N^{s_2s_1} = N^{s_1s_2}$. Therefore $N^{s_2}$ is orchard.
    \item $d=b$ and $c\neq a,b$. As $s_1=(a,b)$ and $s_2=(c,b)$ are cherries in $N$,
          $p_b$ is adjacent to both cherries. Therefore, $p_b$ is a root of
          degree 3 and, again, we obtain the same result if we reduce
          first $s_1$ and then $s_2$ or first $s_2$ and then $s_1$, so $N^{s_2}$
          is orchard.
    \item $c\neq a,b$ and $d\neq a,b$. Then, the reduction of one cherry does not
    affect the other, so again we obtain $N^{s_1s_2}=N^{s_2s_1}$ and
    hence $N^{s_2}$ is orchard.
  \end{itemize}
  In all cases, we concluded that $N^{s_2}$ is orchard.
\end{proof}

\begin{lemma} [{\cite[Lemma 9]{2021JanssenMurakami}}]
  \label{lem:order-9}
  Let $N$ be an orchard $\cL$-network, and let $(a,b)$ be a cherry in $N$. Then,
  there exists a complete reduction sequence $S=s_1\ldots s_k$ for $N$ such
  that $s_i=(a,b)$ or $s_i=(b,a)$, and such that $(a,b)$ is a cherry
  in $N^{s_1\ldots s_j}$ for every $j\in[i-1]$.
\end{lemma}

\begin{proof}
  Assume that there is no sequence with the desired property.
  If $S=s_1\ldots s_k$ is a complete reduction sequence for $N$,
  define $i(S)$ as the smallest index such that $(a,b)$ is not a
  cherry in $N^{s_1\ldots s_{i}}$. By assumption, $i(S)\geq 2$.
  Let $S=s_1\ldots s_k$ be a complete reduction sequence for $N$,
  such that $i(S)$ is maximum. To arrive to a contradiction,
  we will build a complete reduction sequence $S'$ that has either
  the desired property, or $i(S')>i(S)$.
  We denote $N_j=N^{s_1\ldots s_j}$ for every $j\in[k]$ and $i=i(S)$.
  Then by Lemma~\ref{lem:order-6},
  $s_i=(a,c)$ or $s_i=(b,c)$ for some $c\neq a,b$.
  We claim that $(a,b), (a,c)$ and $(b,c)$ are cherries of type (r3)
  in $N_{i-1}$, and $(b,c)$ is a tree cherry. To prove it, let $p_a$, $p_b$ and
  $p_c$ denote the parents of $a$, $b$ and $c$ in $N_{i-1}$, respectively.
  \begin{itemize}
    \item If $s_i=(a,c)$, then we already know that $(a,b)$ and $(a,c)$ are
      cherries in $N_{i-1}$. If $p_a$ is a tree node,
      then necessarily $p_a=p_b=p_c$, so $(a,b),(a,c)$
      and $(b,c)$ are tree cherries of type T(r3). Otherwise, we may have
      that $p_b=p_c$ or not. If they are equal, then $(a,b)$ and $(a,c)$ are
      reticulate cherries of type R(r3), and $(b,c)$ is a tree cherry of
      type T(r3). If $p_b\neq p_c$, then reducing $s_i=(a,c)$
      causes the suppression of $p_a$ such that $(a,b)$ becomes a tree
      cherry in $N_i$, which is a contradiction.
    \item If $s_i=(b,c)$, then we already know that $(a,b)$ and $(b,c)$ are
      cherries in $N_{i-1}$. From $(a,b)$ being a cherry, $p_b$ must be a
      tree node, and so $(b,c)$ must be a tree cherry.
      As $p_b=p_c$ has children $b$, $c$ and either $a$ or $p_a$,
      it is an unresolved root, 
      and so $(a,b)$ and $(b,c)$ are of type (r3). Also, $p_b=p_c$ implies that
      $(a,c)$ is a cherry of the same type as $(a,b)$, concluding the case.
  \end{itemize}

  If $s_i=(a,c)$, then $N_i=N_{i-1}^{(a,b)}$. So replacing $s_i$
  by $s'_i=(a,b)$ in $S$ gives a new complete reduction sequence $S'$ for $N$,
  which has the desired property: $s'_i=(a,b)$
  and $(a,b)$ is a cherry of the network at each previous step.

  If instead $s_i=(b,c)$, then
  $N_{i-1}^{(b,c)}$ is equal to $N_{i-1}^{(c,b)}$ in which $b$ is replaced by $c$,
  as $(b,c)$ is a tree cherry.
  Also, $S'=s_1\ldots s_{i-1}(c,b)s_{i+1}'\ldots s_k'$ is a complete 
  reduction sequence for $N$, where $s_j'$ is obtained by replacing $c$ by $b$
  in $s_j$ for all $j\in\{i+1,\ldots,k\}$.
  Since $(a,b)$ is a cherry in $N_{i-1}^{(c,b)}$, we have that
  $i(S')\geq i+1$.

  In either case, we obtain a contradiction, as claimed.
\end{proof}

\bigskip

We are now ready to prove Proposition~\ref{prop:order-semi},
the analog of Proposition~1 from \cite{2021JanssenMurakami} for
semidirected (binary) networks.

\begin{proof}[Proof of Proposition~\ref{prop:order-semi}]
  Let $(a,b)$ be a cherry in an orchard $\cL$-network $N$.
  By Lemma~\ref{lem:order-9}, there exists a complete reduction sequence $S=s_1\ldots s_k$
  for $N$ such that one of the pairs of $S$ is either $s_i=(a,b)$
  or $s_i=(b,a)$, and $(a,b)$ is a cherry in $N^{s_1\ldots s_j}$ for every $j\in[i-1]$.
  If $i=1$, then $s_2\ldots s_k$ is a complete reduction sequence
  for $N^{(a,b)}$, hence $N^{(a,b)}$ is orchard. Otherwise, we can assume that
  $S$ is such that the index $i$ is the smallest possible,
  with the properties above.

  \smallskip\noindent
  If $s_i=(a,b)$, then, we can apply Lemma~\ref{lem:order-8} $i$ times,
  to get that $N^{(a,b)}$ is orchard.

  \smallskip\noindent
  If $s_i=(b,a)$, then $(a,b)$ and $(b,a)$ are both cherries
  in $N^{s_1\ldots s_j}$ for every $j\in[i-1]$, so they must be
  tree cherries. By Lemma~\ref{lem:order-7}, there exists a complete
  reduction sequence $S'$ for $N$ starting
  with $s_1\ldots s_{i-1}(a,b)$. Then, changing $S$ for $S'$ we can
  apply the previous case and thus $N^{(a,b)}$ is orchard,
  which completes the proof.
\end{proof}

\end{document}